\documentclass[a4paper,10pt, oneside]{article}
\usepackage{amsmath,amssymb,amsthm,mathrsfs,graphicx}
\usepackage{authblk}
\usepackage[font=small,labelfont=md,textfont=it]{caption}
\usepackage{floatrow,float}
\usepackage[titletoc, title]{appendix}
\usepackage{enumitem}
\usepackage[colorlinks,linkcolor=blue,citecolor=blue]{hyperref}
\usepackage{etoolbox}
\usepackage{longtable}
\usepackage{diagbox}
\usepackage{booktabs,makecell,multirow}
\usepackage[capitalise,nosort]{cleveref}
\usepackage{cases,color}
\crefname{equation}{}{}
\crefname{lemma}{Lemma}{Lemmas}
\crefname{theorem}{Theorem}{Theorems}
\crefname{discr}{Discretization}{Discretizations}

\DeclareMathOperator{\D}{D}

\apptocmd{\sloppy}{\hbadness 10000\relax}{}{}

\newcommand{\dual}[1]{( {#1} )}

\newcommand{\nm}[1]{\lVert {#1} \rVert}

\newcommand{\snm}[1]{\lvert {#1} \rvert}

\newcommand{\ssnm}[1]
{
	\left\vert\kern-0.25ex
	\left\vert\kern-0.25ex
	\left\vert
	{#1}
	\right\vert\kern-0.25ex
	\right\vert\kern-0.25ex
	\right\vert
}

\makeatletter
\def\spher@harm#1{%
	\vbox{\hbox{%
			\offinterlineskip
			\valign{&\hb@xt@2\p@{\hss$##$\hss}\vskip.2ex\cr#1\crcr}%
		}\vskip-.36ex}%
}
\def\gshone{\spher@harm{.}}
\def\gshtwo{\spher@harm{.&.}}
\def\gshthree{\spher@harm{.&.&.}}
\let\gsh\spher@harm
\makeatother

\newtheorem{lemma}{Lemma}[section]
\newtheorem{remark}{Remark}[section]
\newtheorem{theorem}{Theorem}[section]

\makeatletter\def\@captype{table}\makeatother

\begin{document}
	\title{
		\Large \bf Discontinuous Galerkin method for  a distributed optimal control problem governed by a
		time fractional diffusion equation
		\thanks
		{Tao Wang was supported in part by the China Postdoctoral Science Foundation (2019M66294),
			Binjie Li was supported in part by the National Natural Science
			Foundation of China (11901410), and Xiaoping Xie
			was supported in part by the National Natural Science Foundation of
			China  (11771312).
		}
	}
	\author[1]{Tao Wang\thanks{ tao\_wang@m.scnu.edu.cn; wangtao5233@hotmail.com}}
	\author[2]{Binjie Li\thanks{Corresponding author: libinjie@scu.edu.cn}}
	\author[2]{Xiaoping Xie\thanks{xpxie@scu.edu.cn}}
	\affil[1]{South China Research Center for Applied Mathematics and Interdisciplinary Studies, South China Normal University, Guangzhou 510631, China}
	\affil[2]{School of Mathematics, Sichuan University, Chengdu 610064, China}

	\date{}
	\maketitle
	\begin{abstract}
		This paper is devoted to  the numerical analysis of  a control constrained distributed optimal control problem subject to
		a time fractional diffusion equation with non-smooth initial data. The solutions of state and co-state  are decomposed into singular  and regular parts, and some growth estimates are obtained for the singular parts.  By following the variational discretization concept,   a full discretization is applied to the corresponding state and co-state equations by using linear conforming finite element method in space and piecewise constant discontinuous Galerkin method in time.  Error estimates are  derived by employing the  growth estimates. In particular, graded temporal grids are adopted
		to obtain the first-order temporal accuracy.  Finally, numerical experiments are performed to verify the theoretical results.

	\end{abstract}


	\section{Introduction}
	Let $ \Omega \subset \mathbb R^d $ ($d=1,2,3$) be a convex polytope, and assume
	that $ 0 < \alpha < 1 $, $ -\infty < u_* < u^* < \infty $, and $ 0 < \nu, T <
	\infty $. We consider the following distributed optimal control problem:
	\begin{equation}
	\label{eq:model}
	\min\limits_{
		\substack{
			u \in U_{\text{ad}} \\
			y \in L^2(0,T;L^2(\Omega))
		}
	} J(y,u) = \frac12 \nm{y - y_d}_{L^2(0,T;L^2(\Omega))}^2 +
	\frac\nu2 \nm{u}_{L^2(0,T;L^2(\Omega))}^2,
	\end{equation}
	subject to the state equation
  \begin{equation}
    \label{eq:state}
    \begin{cases}
      \big( \D_{0+}^\alpha (y-y_0) \big)(t) - \Delta y(t) = u(t),
      & 0 < t \leqslant T, \\
      y(0) = y_0.
    \end{cases}
  \end{equation}
	Here, $ \Delta $ is the realization of the Laplace operator with homogeneous
	Dirichlet boundary condition in $ L^2(\Omega) $, $ \D_{0+}^\alpha $ is the
	Riemann-Liouville fractional differential operator of order $ \alpha $, $ y_0
	\in L^2(\Omega) $ and $ y_d \in L^2(0,T;L^2(\Omega)) $ are given, and
	\begin{align*}
	U_{\text{ad}} &:= \left\{
	v \in L^2(0,T;L^2(\Omega)):\
	u_* \leqslant v \leqslant u^*
	\text{ a.e.~in } \Omega \times (0,T)
	\right\}.
	\end{align*}

	The optimal control problem \cref{eq:model} subject to an elliptic or heat
	equation is a classic problem, which has been thoroughly studied both in
	theoretical and numerical aspects; see, e.g. \cite{M2008,lion71,
		troltzsch2010optimal} for theoretical analysis and \cite{
		Gongs16ada,Dmitriy2013Optimal,Li02ada, Wang-Y-X2019, Wang-Y-X2019a} for
	finite element analysis. In general, there are mainly two discretization
	concepts for this problem (\cite{R2006Optimal,M2008}): direct and variational discretizations. The
	difference between these two concepts is that in the variational discretization,
	the control is implicitly discretized by the $L^2$ projection of the discrete co-state into the admissible
	set $U_{ad}$. Since the control may has singularity near the boundary of active
	set, the variational discretization is easier to obtain high accuracy than the
	direct discretization. However, it should be pointed out that the resultant
	discrete system of variational discretization is generally more difficult to
	solve \cite{Hinze2009Variational,R2017Mass}.

	The state equation $\cref{eq:state}$ is a fractional diffusion equation, which is used to model some physical processes like subdiffusion \cite{bouchaud1990anomalous} and water movement in soils \cite{Schumer2009Fractional}. There are many   methods to solve this equation, including finite difference methods \cite{Gao2011,Jin2016,Langlands2005, Lin2007,Lubich1996,  
		Yuste2005,Zeng2013},   spectral methods \cite{Li2009,yang2016spectral},
	finite element methods
	\cite{Li2019SIAM,McLean2009Convergence,Mustapha2009Discontinuous,
		Mustapha2011Piecewise,
		Mustapha2012Uniform,Mustapha2014A} and so on.

	In recent years,  fractional optimal control problems have  attracted more and more research interest \cite{antil2015,antil2016a,Max2019,jin2017pointwise-in-time, fopt1,Ye2015A, Zhang;Liu;Zhou2019, Zhou2016Finite,Zhou;Zhang2018}.
	For problem \cref{eq:model},
	Zhou and Gong \cite{Zhou2016Finite} employed the conforming linear finite element method and L1 scheme for spatial and temporal discretizations, respectively, and    obtained optimal convergence results for the spatially semi-discrete approximation. By using the conforming linear finite element method in space and the L1 scheme/backward Euler convolution quadrature in time, Jin et al. \cite{jin2017pointwise-in-time} gave the first error estimate of the fully discrete scheme  with   $y_0=0$, which is nearly optimal with respect to the regularity.   
	In \cite{Max2019, Zhang;Liu;Zhou2019} error estimates were derived  for  fully discrete finite element approximations  for problem $\cref{eq:model}$ with a variant state equation like
	\begin{equation*}
	y'- \D_{0+}^{1-\alpha}\Delta y = f+u. 
	\end{equation*}
	%
	We note that the above works \cite{Max2019,jin2017pointwise-in-time,Zhang;Liu;Zhou2019,Zhou2016Finite} all focus  on the case $y_0=0$,  and their analyses are  based on uniform or quasi-uniform temporal grids. However, the situation will  be quite subtle when considering nonvanishing $y_0$. In fact, the nonvanishing initial value  may cause  essential singularities  (cf. \cref{thm:regu-y-p}), which can not be handled well by  using uniform or quasi-uniform temporal grids. On the other hand, the non-vanishing $y_d(T)$ may also cause singularities.  Fortunately,   all these  singularities can be dealt with by using  special graded temporal grids (cf. \cref{th:ape}).

  In this paper, for a full discretization using the conforming linear finite
  element method in space and the piecewise constant discontinuous Galerkin
  method in time, we provide the first numerical analysis of problem
  \cref{eq:model} with nonvanishing $ y_0 $. Moreover, for the case with $ y_0
  \in \dot H^{2r}(\Omega) $ with $0<r<\min\{1,\frac{1-\alpha}{\alpha}\}$ and $
  y_d \in H^1(0,T;L^2(\Omega)) $, 
  we have the following decompositions of the control
  $u$, the state $y$, and the co-state $p$:
	\begin{align*}
	u=u_1+u_2, \quad y = y_1 + y_2,  \quad
	p= p_1 + p_2,
	\end{align*}
	with the regularity estimates
	\begin{align*}
	&\nm{u_1}_{{}_0H^1(0,T;L^2(\Omega))} \leqslant C,\\
	&\nm{u_2'(t)}_{L^2(\Omega)} \leqslant
	C \big( t^{\alpha r +\alpha -1}+ (T-t)^{\alpha-1} \big),\\
	& \nm{y_1}_{{}_0H^{1+\alpha}(0,T;L^2(\Omega))} +
	\nm{y_1}_{{}_0H^1(0,T;\dot H^2(\Omega)}
	\leqslant C, \\
	& \nm{p_1}_{{}^0H^{1+\alpha}(0,T;L^2(\Omega))} +
	\nm{p_1}_{{}^0H^1(0,T;\dot H^2(\Omega))}
	\leqslant C, \\
	& \nm{y_2'(t)}_{L^2(\Omega)}
	\leqslant C \big( t^{\alpha r-1} + \omega_2(T-t) \big),
	\quad 0 < t < T, \\
	& \nm{y_2'(t)}_{\dot H^1(\Omega)}
	\leqslant C \big( t^{\alpha r - \alpha/2-1} + \omega_1(T-t) \big),
	\quad 0 < t < T, \\
	& \nm{p_2'(t)}_{L^2(\Omega)}
	\leqslant C \big( t^{\alpha r +\alpha -1} + (T-t)^{\alpha-1} \big),
	\quad 0 < t < T, \\
	& \nm{p_2'(t)}_{\dot H^1(\Omega)}
	\leqslant C \big(  t^{\alpha r +\alpha/2 -1} + (T-t)^{\alpha/2-1} \big),
	\quad 0 < t < T,
	\end{align*}
	where  \begin{align*}
	\omega_1(t) = \begin{cases}
	1+\frac{t^{3\alpha/2 -1}}{\snm{\alpha(2-3\alpha)}}
	& \text{ if }  \alpha \ne 2/3, \\
	\snm{\ln t} & \text{ if } \alpha = 2/3,
	\end{cases} \\
	\omega_2(t) = \begin{cases}
	1+\frac{t^{2\alpha-1}}{\snm{\alpha(1-2\alpha)}}
	& \text{ if } \alpha \ne 1/2, \\
	\snm{\ln t} & \text{ if } \alpha = 1/2,
	\end{cases}
	\end{align*} and $ C $ is a generic positive constant depending only on $ \alpha $, $r$, $ \nu $, $ u_* $,
	$ u^* $, $ y_0 $, $ y_d $, $ T $, and $ \Omega $.
	By the above estimates, we obtain first-order temporal accuracy and $\min\{2,1/\alpha+2r\}$-order spatial  accuracy   on graded temporal grids.

	The rest of this paper is organized as follows. \cref{sec:pre} introduces several Sobolev spaces and the Riemann-Liouville fractional calculus
	operators. \cref{sec:regu} investigates the regularity of problem
	\cref{eq:model}.   \cref{sec:conv}  carries out the convergence analysis for the discontinuous Galerkin method.  Finally, \cref{sec:numer} provides several numerical experiments to
	confirm the theoretical results.
\section{Preliminaries} \label{sec:pre}
In this paper, we introduce the following conventions: if $ D \subset \mathbb
R^l(l=1,2,3,4) $ is Lebesgue measurable, then define $ \dual{v,w}_D := \int_D v
\cdot w $ for scaler or vector valued functions $ v $ and $ w $, and if $X$ is a
Banach space, then $\dual{\cdot,\cdot}_X$ means the duality pairing between $
X^* $ (the dual space of $X$) and $X$; the notation $ C_\times $ means a
positive constant depending only on its subscript(s), and its value may differ
at each occurrence. Let $ H^\gamma(D) $ ($ -\infty < \gamma < \infty $) and $
H_0^\gamma(D) $ ($ 0<\gamma<\infty $) be the usual $\gamma$-th order Sobolev
spaces on $ D $ with norm $\nm{\cdot}_{H^\gamma(D)}$ and seminorm
$\snm{\cdot}_{H^\gamma(D)}$. In particular, $H^0(D)=L^2(D).$

	\medskip  \noindent \textbf{Sobolev Spaces.} Assume that $ -\infty < a < b < \infty $ and $ X $ is a Hilbert space. For each
	$ m \in \mathbb N $ and $ 1 \leqslant q \leqslant \infty $, define
	\begin{align*}
	{}_0W^{m,q}(a,b;X) & := \left\{
	v\in W^{m,q}(a,b;X):\
	v^{(k)}(a)=0, \,\, 0 \leqslant k < m
	\right\}, \\
	{}^0W^{m,q}(a,b;X) & := \left\{
	v\in W^{m,q}(a,b;X):\
	v^{(k)}(b)=0, \,\, 0 \leqslant k < m
	\right\},
	\end{align*}
	where $ W^{m,q}(a,b;X) $ is the usual vector valued Sobolev space and $ v^{(k)}
	$ is the $ k $-th weak derivative of $ v $. We equip the above two spaces with
	the norms
	\begin{align*}
	\nm{v}_{{}_0W^{m,q}(a,b;X)} &:= \nm{v^{(m)}}_{L^q(a,b;X)}
	\quad \forall v \in {}_0W^{m,q}(a,b;X), \\
	\nm{v}_{{}^0W^{m,q}(a,b;X)} &:= \nm{v^{(m)}}_{L^q(a,b;X)}
	\quad \forall v \in {}^0W^{m,q}(a,b;X),
	\end{align*}
	respectively. For any $ m \in \mathbb N_{>0} $ and $ 0 < \theta < 1 $, define
	\begin{align*}
	W^{m-1+\theta,q}(a,b;X) &:=
	\left(W^{m-1,q}(a,b;X), W^{m,q}(a,b;X) \right)_{\theta,q}, \\
	{}_0W^{m-1+\theta,q}(a,b;X) &:=
	\left({}_0W^{m-1,q}(a,b;X), \ {}_0W^{m,q}(a,b;X)\right)_{\theta,q}, \\
	{}^0W^{m-1+\theta,q}(a,b;X) &:=
	\left( {}^0W^{m-1,q}(a,b;X), \ {}^0W^{m,q}(a,b;X) \right)_{\theta,q},
	\end{align*}
	where $ (A,B)_{\theta,q} $ denotes the real interpolation space of  two Banach spaces,   $ A $ and $ B $, constructed by the  $ K $-method
	\cite{Tartar2007}.
	In addition, for $q=2$ and $ 0 \leqslant \beta < \infty $,
	denote
	$ H^\beta(a,b;X) :=W^{\beta,2}(a,b;X),$
	$ {}_0H^\beta(a,b;X) := {}_0W^{\beta,2}(a,b;X),$ and $  {}^0H^\beta(a,b;X) := {}^0W^{\beta,2}(a,b;X).$
	We also need the   space
	\[
	W_\text{loc}^{1,\infty}(a,b;X) :=
	\left\{
	v: (a,b) \to X:\
	v \in W^{1,\infty}(c,d;X)
	\,\text{ for all } a < c < d < b
	\right\}.
	\]
	\begin{remark}
		If $ 0 < \theta < 1 $ and $ 1 \leqslant q < \infty $ satisfy $ \theta q<1 $,
		then
		\[
		W^{\theta, q}(a,b;X) = {}_0W^{\theta, q}(a,b;X) =
		{}^0W^{\theta, q}(a,b;X)
		\]
		with equivalent norms.
	\end{remark}

	Let $ \Delta $ be the realization of the Laplace operator
	with homogeneous Dirichlet boundary condition in $ L^2(\Omega) $. For any $
	-\infty< r < \infty $, define
	\[
	\dot H^r(\Omega) := \{ (-\Delta)^{-r/2} v:\ v \in L^2(\Omega) \}
	\]
	and endow this space with the norm
	\[
	\nm{v}_{\dot H^r(\Omega)} := \nm{(-\Delta)^{r/2} v}_{L^2(\Omega)}
	\quad \forall v \in \dot H^\beta(\Omega).
	\]

	\begin{remark}
		For $r \in [0, 1]\setminus\{0.5\}$, $\dot H^{r}(\Omega) =
		H_0^r(\Omega) $ holds with equivalent norms, and for $1<r\leqslant 2$, the space
		$\dot H^{r}(\Omega)$ is continuously embedded into $H^{r}(\Omega)$.
	\end{remark}
	\medskip \noindent \textbf{Fractional calculus operators.}
	For $  \gamma > 0 $,  the left-sided and right-sided Riemann-Liouville fractional integral operators of order $\gamma$ are defined respectively by
	\begin{align*}
	\big(\D_{0+}^{-\gamma} v\big)(t) &:=
	\frac1{ \Gamma(\gamma) }
	\int_0^t (t-s)^{\gamma-1} v(s) \, \mathrm{d}s, \quad 0 < t < T, \\
	\big(\D_{T-}^{-\gamma} v\big)(t) &:=
	\frac1{ \Gamma(\gamma) }
	\int_t^T (s-t)^{\gamma-1} v(s) \, \mathrm{d}s, \quad 0 < t < T,
	\end{align*}
	for all $ v \in L^1(0,T;X) $, where $ \Gamma(\cdot) $ is the gamma function. In
	addition, let $ \D_{0+}^0 $ and $ \D_{T-}^0 $ be the identity operator on $
	L^1(0,T;X) $. Then for $0< \gamma \leqslant 1$, the left-sided and right-sided Riemann-Liouville fractional differential operators of order  $\gamma$ are defined respectively by
	\begin{align*}
	\D_{0+}^\gamma v & := \D \D_{0+}^{\gamma-1}v, \\
	\D_{T-}^\gamma v & := -\D\D_{T-}^{\gamma-1}v,
	\end{align*}
	for all $ v \in L^1(0,T;X) $, where $ \D $ is the first-order differential operator
	in the distribution sense.

	\begin{lemma}[\cite{ervin06}] 
		\label{lem:dual}
		If $v\in  H^{\alpha/2}(0,T),$ then
		\begin{align*}
		\dual{ \D_{0+}^{\alpha/2} v, \D_{T-}^{\alpha/2} v }_{(0,T)}  \geqslant C_{\alpha,T} \nm{v}_{ H^{\alpha/2}(0,T)}^2.
		\end{align*}
		Moreover, if $v,w \in  H^{\alpha/2}(0,T)$, then
		\begin{align*}
		&	\dual{\D_{0+}^{\alpha/2} v, \D_{T-}^{\alpha/2} w}_{(0,T)} \leqslant C_{\alpha,T} \nm{v}_{H^{\alpha/2}(0,T)} \nm{w}_{H^{\alpha/2}(0,T)},  \\
		&\dual{\D_{0+}^{\alpha} v,w}_{H^{\alpha/2}(0,T)} = 	\dual{\D_{0+}^{\alpha/2} v,\D_{T-}^{\alpha/2}w}_{(0,T)}= 	\dual{\D_{T-}^{\alpha}w,v}_{H^{\alpha/2}(0,T)}.
		\end{align*}
	\end{lemma}


\section{Regularity}
\label{sec:regu}
For any $ g \in L^q(0,T;L^2(\Omega)) $ with $ 1 < q < \infty $, define $Sg$
(cf.~\cref{sec:regu-S}) such that $(\D_{0+}^{\alpha}-\Delta)Sg=g.$ From
\cref{lem:regu-lp} we summarize several properties of $S$ as follows: %
(cf.~\cref{lem:regu-lp}):
\begin{small}
\begin{itemize}
  \item for any $ g \in {}_0W^{\beta,q}(0,T;L^2(\Omega)) $ with $ \beta \in (0,2]
    \setminus \{1-\alpha,2-\alpha\} $ and $ 1 < q < \infty $,
    \begin{equation}
      \label{eq:Sg-real}
      \nm{Sg}_{{}_0W^{\alpha+\beta,q}(0,T;L^2(\Omega))} +
      \nm{Sg}_{{}_0W^{\beta,q}(0,T;\dot H^2(\Omega))}
      \leqslant C_{\alpha,\beta,q} \nm{g}_{{}_0W^{\beta,q}(0,T;L^2(\Omega))};
    \end{equation}
  \item for any $ g \in {}_0H^\beta(0,T;L^2(\Omega)) $ with $ 0 \leqslant \beta < \infty
    $,
    \begin{equation}
      \label{eq:Sg-real-complex}
      \nm{Sg}_{{}_0H^{\alpha+\beta}(0,T;L^2(\Omega))} +
      \nm{Sg}_{{}_0H^\beta(0,T;\dot H^2(\Omega))}
      \leqslant C_{\alpha,\beta} \nm{g}_{{}_0H^\beta(0,T;L^2(\Omega))}.
    \end{equation}
\end{itemize}
\end{small}
Symmetrically, for any $ g \in L^q(0,T;L^2(\Omega)) $ with $ 1 < q < \infty $,
define $S^*g$ such that $(\D_{T-}^{\alpha}-\Delta)S^*g=g.$ Similar to $S$, there hold following properties of $S^*$:
\begin{small}
\begin{itemize}
  \item for any $ g \in {}^0W^{\beta,q}(0,T;L^2(\Omega)) $ with $ \beta \in (0,2] \setminus
    \{1-\alpha,2-\alpha\} $ and $ 1 < q < \infty $,
    \begin{equation}
      \label{eq:S^*g-real}
      \nm{S^*g}_{{}^0W^{\alpha+\beta,q}(0,T;L^2(\Omega))} +
      \nm{S^*g}_{{}^0W^{\beta,q}(0,T;\dot H^2(\Omega))}
      \leqslant C_{\alpha,\beta,q} \nm{g}_{{}^0W^{\beta,q}(0,T;L^2(\Omega))};
    \end{equation}
  \item for any $ g \in {}^0H^\beta(0,T;L^2(\Omega)) $ with $ 0 \leqslant \beta < \infty
    $,
    \begin{equation}
      \label{eq:S^*g-real-complex}
      \nm{S^*g}_{{}^0H^{\alpha+\beta}(0,T;L^2(\Omega))} +
      \nm{S^*g}_{{}^0H^\beta(0,T;\dot H^2(\Omega))}
      \leqslant C_{\alpha,\beta} \nm{g}_{{}^0H^{\beta}(0,T;L^2(\Omega))}.
    \end{equation}
\end{itemize}
\end{small}
In addition, by the definitions of $ S $ and $ S^* $, \cref{eq:Sg-real-complex},
\cref{eq:S^*g-real-complex} and \cref{lem:dual}, we obtain that, for any $ v, w
\in L^2(0,T;L^2(\Omega)) $,
\begin{align}
  (Sv,w)_{\Omega \times (0,T)} &=
  \big(Sv, (\D_{T-}^\alpha - \Delta)S^*w \big)_{\Omega \times (0,T)}
  \notag \\
  &=
  \big((\D_{0+}^\alpha - \Delta)Sv, S^*w\big)_{\Omega \times (0,T)}
  \notag \\
  &= (v, S^*w)_{\Omega \times (0,T)}.
  \label{eq:S-S*}
\end{align}

Assuming that $ y_0 \in L^2(\Omega) $ and $ y_d \in L^2(0,T;L^2(\Omega)) $, we
call $ u \in U_{\text{ad}} $ a solution to problem \cref{eq:model} if $ u $
solves the minimization problem
\begin{equation}
  \label{eq:weak_sol}
  \min\limits_{u \in U_{\text{ad}}} J(u) =
  \frac12 \nm{
    S(u+\D_{0+}^\alpha y_0)- y_d
  }_{ L^2(0,T;L^2(\Omega))}^2 +
  \frac\nu2 \nm{u}_{L^2(0,T;L^2(\Omega))}^2.
\end{equation}

By \cref{eq:S-S*}, a routine argument gives the following theorem
(cf.~\cite{M2008,troltzsch2010optimal}).
\begin{theorem}
  \label{thm:optim_cond}
  Problem \cref{eq:model} admits a unique solution $ u \in U_{\text{ad}} $, and
  \begin{equation}
    \label{eq:optim_cond}
    \dual{
      S^*\big( S(u+\D_{0+}^\alpha y_0) - y_d \big) + \nu u, v - u
    }_{\Omega \times (0,T)} \geqslant 0
  \end{equation}
  for all $ v \in U_\text{ad} $.
\end{theorem}
In the rest of this paper, we use $ u $ to denote the solution of problem
\cref{eq:model} and use
\begin{equation}
  \label{eq:y-p-def}
  y:=S(u+\D_{0+}^\alpha y_0)\ \text{ and } \
  p:=S^*\big( y - y_d \big)
\end{equation}
to denote the corresponding state and co-state, respectively.

The main task of this section is to prove the following two theorems.
\begin{theorem}
  \label{thm:regu-u}
  Assume that $ y_0 \in \dot H^{2r}(\Omega) $ with
  $0<r<\min\{1,\frac{1-\alpha}{\alpha} \}$, and $ y_d \in H^1(0,T;L^2(\Omega))
  $. There exists a decomposition
  \[
    u= u_1+ u_2,
  \]
  where
  \begin{equation}
    \label{eq:u1}
    \nm{u_1}_{{}_0H^1(0,T;L^2(\Omega))} \leqslant C
  \end{equation}
  and $ u_2 \in W_\text{loc}^{1,\infty}(0,T;L^2(\Omega)) $ satisfies that
  \begin{equation}
    \label{eq:u2}
    \nm{u_2'(t)}_{L^2(\Omega)} \leqslant
    C \big( t^{\alpha r +\alpha -1}+ (T-t)^{\alpha-1} \big),
    \quad\text{a.e.}~0 < t < T,
  \end{equation}
  where  $ C $ is a positive constant depending only on $ \alpha $, $r$, $ \nu $,
  $ u_* $, $ u^* $, $ y_0 $, $ y_d $, $ T $ and $ \Omega $.
\end{theorem}


	\begin{theorem}
		\label{thm:regu-y-p}
		Assume that $ y_0 \in \dot H^{2r}(\Omega) $  with $0<r<\min\{1,\frac{1-\alpha}{\alpha} \}$, and $ y_d \in H^{1}(0,T;L^2(\Omega))
		$.  There exist decompositions
		\begin{align*}
		y= y_1 + y_2, \quad  p= p_1 + p_2,
		\end{align*}
		where
		\begin{align}
		& y_1 \in {}_0H^{1+\alpha}(0,T;L^2(\Omega))
		\bigcap {}_0H^1(0,T;\dot H^2(\Omega)), \\
		& p_1 \in {}^0H^{1+\alpha}(0,T;L^2(\Omega))
		\bigcap {}^0H^1(0,T;\dot H^2(\Omega)) \\
		& \text{and} \quad y_2, p_2 \in
	 C^1((0,T);\dot H^1(\Omega)).
		\end{align}
		Moreover,
		\begin{align}
		& \nm{y_1}_{{}_0H^{1+\alpha}(0,T;L^2(\Omega))} +
		\nm{y_1}_{{}_0H^1(0,T;\dot H^2(\Omega)}
		\leqslant C, \\
		& \nm{p_1}_{{}^0H^{1+\alpha}(0,T;L^2(\Omega))} +
		\nm{p_1}_{{}^0H^1(0,T;\dot H^2(\Omega))}
    \leqslant C, \label{eq:regu-p1} \\
		& \nm{y_2'(t)}_{\dot H^1(\Omega)}
		\leqslant C \big( t^{\alpha r - \alpha/2-1} + \omega_1(T-t) \big),
		\quad 0 < t < T, \\
		& \nm{y_2'(t)}_{L^2(\Omega)}
		\leqslant C \big( t^{\alpha r-1} + \omega_2(T-t) \big),
		\quad 0 < t < T, \\
		& \nm{p_2'(t)}_{\dot H^1(\Omega)}
		\leqslant C \big( t^{\alpha r + \alpha/2-1} + (T-t)^{\alpha/2-1} \big),
    \quad 0 < t < T, \label{eq:regu-p2-h1} \\
		& \nm{p_2'(t)}_{L^2(\Omega)}
		\leqslant C \big(t^{\alpha r +\alpha -1}+ (T-t)^{\alpha-1} \big),
    \quad 0 < t < T. \label{eq:regu-p2}
		\end{align}
		The above $ C $ is a positive constant depending only on $ \alpha $, $r$, $ \nu $,
		$ u_* $, $ u^* $, $ y_0 $, $ y_d $, $ T $ and $ \Omega $, and for any $ t > 0 $,
		\begin{align}
		\omega_1(t) := \begin{cases}
		1+\frac{t^{3\alpha/2 -1}}{\snm{\alpha(2-3\alpha)}}
		& \text{ if }  \alpha \ne 2/3, \\
		\snm{\ln t} & \text{ if } \alpha = 2/3,
		\end{cases} \label{eq:omega1} \\
		\omega_2(t) := \begin{cases}
		1+\frac{t^{2\alpha-1}}{\snm{\alpha(1-2\alpha)}}
		& \text{ if } \alpha \ne 1/2, \\
		\snm{\ln t} & \text{ if } \alpha = 1/2.
		\end{cases} \label{eq:omega2}
		\end{align}

	\end{theorem}

	\begin{remark} The results of \cref{thm:regu-u,thm:regu-y-p} can be easily extended to the case $y_0 \in \dot H^{2r}(\Omega)$ with $r\geqslant \min\{1,\frac{1-\alpha}{\alpha}\}$.
	\end{remark}


\subsection{Proofs of \texorpdfstring{\cref{thm:regu-u,thm:regu-y-p}}{}} 

For $ g \in L^1(0,T;L^2(\Omega)) $, we have that \cite{Jin2018}
\begin{align}
  (Sg)(t) &= \int_0^t E(s) g(t-s) \, \mathrm{d}s,
  \qquad \, \text{a.e.}~0 < t < T, \label{eq:Sg} \\
  (S^*g)(t) &= \int_t^T E(s-t) g(s) \, \mathrm{d}s,
  \qquad \, \text{a.e.}~0 < t < T, \label{eq:S*g}
\end{align}
where, for each $ 0 < s \leqslant T $,
\begin{equation*}
  E(s) := \frac1{2\pi i}
  \int_0^\infty e^{-rs} \big(
    (r^\alpha e^{-i\alpha\pi} - \Delta)^{-1} -
    (r^\alpha e^{i\alpha\pi} - \Delta)^{-1}
  \big) \, \mathrm{d}r.
\end{equation*}
\begin{lemma}[\cite{Jin2018}]
  \label{lem:E}
  The function $ E $ is an $ \mathcal L(L^2(\Omega),\dot H^1(\Omega)) $-valued
  analytic function on $ (0,\infty) $, and
  \begin{equation*}
    \nm{E(t)}_{\mathcal L(L^2(\Omega))} +
    t^{\alpha/2} \nm{E(t)}_{\mathcal L(L^2(\Omega), \dot H^1(\Omega))}
    \leqslant C_\alpha t^{\alpha-1}, \quad t > 0.
  \end{equation*}
\end{lemma}

In the rest of this subsection, for convenience we will always assume that $ y_d
\in H^{1}(0,T;L^2(\Omega)) $ and $ y_0 \in \dot H^{2r}(\Omega) $, where
$0<r<\min\{1,\frac{1-\alpha}{\alpha}\}$.

\begin{lemma} 
  \label{lem:SgS*g}
  Assume that $ g \in W_\text{loc}^{1,\infty}(0,T;L^2(\Omega)) $ and $ A $ is a
  positive constant. If
  \begin{equation}
    \label{eq:SgS*g-g}
    \nm{g'(t)}_{L^2(\Omega)} \leqslant
    A \big( t^{\alpha r + \alpha -1}+ (T-t)^{\alpha-1} \big),
    \quad\text{a.e.}~ 0 < t < T,
  \end{equation}
  then $ Sg \in C^1((0,T);\dot H^1(\Omega)) $ and, for any $ 0 < t < T $,
  \begin{align}
    \nm{(Sg)'(t)}_{L^2(\Omega)} & \leqslant
    C_{\alpha,T} \big( A + \nm{g(0)}_{L^2(\Omega)} \big)
    \big( t^{ \alpha -1} + \omega_2(T-t) \big),
    \label{eq:90} \\
    \nm{(Sg)'(t)}_{\dot H^1(\Omega)} & \leqslant
    C_{\alpha,T} \big( A + \nm{g(0)}_{L^2(\Omega)} \big)
    \big( t^{\alpha/2-1} + \omega_1(T-t) \big).
    \label{eq:91}
  \end{align}
  If
  \begin{equation*}
    \nm{g'(t)}_{L^2(\Omega)} \leqslant
    A\big( t^{\alpha r-1} + \omega_2(T-t) \big),
    \quad\text{a.e.}~0 < t < T,
  \end{equation*}
  then $ S^*g \in C^1((0,T);\dot H^1(\Omega)) $ and, for any $ 0 < t < T $,
  \begin{align*}
    \nm{(S^*g)'(t)}_{L^2(\Omega)} & \leqslant
    C_{\alpha,T} \big( A + \nm{g(T)}_{L^2(\Omega)} \big)
    \big( t^{\alpha r +\alpha -1} + (T-t)^{\alpha -1} \big), \\
    \nm{(S^*g)'(t)}_{\dot H^1(\Omega)} & \leqslant
    C_{\alpha,T} \big( A + \nm{g(T)}_{L^2(\Omega)} \big)
    \big( t^{ \alpha r+\alpha/2-1} + (T-t)^{\alpha/2 -1} \big).
  \end{align*}
  The above $\omega_1$ and $\omega_2$ are defined by \cref{eq:omega1,eq:omega2},
  respectively.
\end{lemma}
\begin{proof}
  By \cref{eq:Sg}, \cref{eq:SgS*g-g} and \cref{lem:E}, a straightforward
  computation gives
  \begin{equation}
    \label{eq:lxy}
    (Sg)'(t) = E(t)g(0) + \int_0^t E(s) g'(t-s) \, \mathrm{d}s,
    \quad 0 < t < T,
  \end{equation}
  and, by the techniques in the proof of \cite[Theorem 2.6]{Diethelm2010}, it
  is easy to verify that $ Sg \in C^1((0,T); \dot H^1(\Omega)) $.
  Furthermore, by \cref{eq:SgS*g-g}, \cref{eq:lxy,lem:E},
  \begin{small}
  \begin{align*}
    \nm{(Sg)'(t)}_{L^2(\Omega)}
    &  \leqslant
    C_\alpha \Big(
      t^{\alpha-1} \nm{g(0)}_{L^2(\Omega)} +
      A \int_0^t s^{\alpha-1}\big(
        (t-s)^{\alpha r + \alpha  -1}\!+\! (T\!-\!t\!+\!s)^{\alpha -1}
      \big) \, \mathrm{d}s
    \Big) \\ & =
    C_\alpha \Big(
      t^{\alpha-1} \nm{g(0)}_{L^2(\Omega)} +
      A t^{\alpha r +2\alpha-1} \int_0^1 x^{\alpha-1}
      (1-x)^{\alpha r + \alpha  -1} \mathrm{d}x \\
      &\qquad \qquad \qquad \qquad \qquad \quad
      \ \ \, {} + A(T-t)^{2\alpha-1}  \int_{0}^{t/(T-t)} x^{\alpha-1}(1+x)^{\alpha-1} \, \mathrm{d}x
    \Big) \\
    & \leqslant C_{\alpha,T} \big( A + \nm{g(0)}_{L^2(\Omega)} \big)
    \Big( t^{\alpha  -1} + \omega_2(T-t) \Big)
  \end{align*}
  \end{small}
  and
  \begin{small}
  \begin{align*}
    \nm{(Sg)'(t)}_{\dot H^1(\Omega)} & \leqslant
    C_\alpha \! \Big(
      t^{\alpha/2-1} \! \nm{g(0)}_{L^2(\Omega)} \!+\!
      A\! \int_0^t \! s^{\alpha/2-1} \big(
        (t-s)^{\alpha r +\alpha-1} \!+\! (T\!-\!t\!+\!s)^{\alpha-1}
      \big)  \mathrm{d}s
    \Big) \\ & =
    C_\alpha \! \Big(
      t^{\alpha/2-1} \! \nm{g(0)}_{L^2(\Omega)} \!+\!
      A t^{\alpha r+3\alpha/2-1} \int_0^1 \! x^{\alpha/2-1} \big(
        (1-x)^{\alpha r +\alpha-1}   \mathrm{d}x \\
        & \qquad \qquad \qquad \qquad \ \
        \qquad {} + A(T-t)^{3\alpha/2-1} \int_{0}^{t/(T-t)} x^{\alpha/2-1}(1+x)^{\alpha-1} \, \mathrm{d}x
      \Big)\\
      & \leqslant C_{\alpha,T} \big( A + \nm{g(0}_{L^2(\Omega)} \big)
      \big( t^{\alpha/2-1} + \omega_1(T-t) \big)
    \end{align*}
    \end{small}
    for all $ 0 < t < T $. This proves estimates \cref{eq:90,eq:91}. Since the rest of
    this lemma can be proved analogously, this completes the proof.
  \end{proof}

	From \cref{eq:optim_cond} it follows that
	\begin{equation}
	\label{eq:u}
	u = f(p),
	\end{equation}
	where $ f: \mathbb R \to \mathbb R $ is defined by
	\begin{equation*}
	f(v) := \begin{cases}
	u^* & \text{ if } v < -\nu u^*, \\
	-v/\nu & \text{ if } -\nu u^* \leqslant v \leqslant -\nu u_*, \\
	u_* & \text{ if } v > -\nu u_*.
	\end{cases}
	\end{equation*}
  We set
  \[
    f'(r) := \begin{cases}
      0 & \text{ if } r \leqslant -\nu u^*, \\
      -1/\nu & \text{ if  } -\nu u^* < r < \nu u_*, \\
      0 & \text{ if } r \geqslant -\nu u_*.
    \end{cases}
  \]
  For any $ v \in W^{1,1}(0,T;L^2(\Omega)) $, from \cite[Thoerem
  7.8]{Gilbarg2001} we conclude that $ f(v) \in W^{1,1}(0,T;L^2(\Omega)) $ and
  \[
    (f(v))'(t) = f'(v(t)) v'(t), \quad 0 < t < T.
  \]
  Furthermore, applying \cite[Lemma 28.1]{Tartar2007} yields the following
  interpolation result.
	\begin{lemma} 
		\label{lem:fv}
		If $ v \in W^{\beta,q}(0,T;L^2(\Omega)) $ with $ 0 \leqslant \beta \leqslant 1 $
		and $ 1 < q < \infty $, then
		\begin{equation*}
		\nm{f(v)}_{W^{\beta,q}(0,T;L^2(\Omega))}
		\leqslant C_{\nu,u_*,u^*,q,T,\Omega} \big(
		1 + \nm{v}_{W^{\beta,q}(0,T;L^2(\Omega))}
		\big).
		\end{equation*}
	\end{lemma}

  \begin{lemma} 
    \label{lem:u-to-u}
    If $ u \in W^{\beta,q_0}(0,T;L^2(\Omega)) $ with $ \beta \in (0,1-\alpha)
    \setminus \{1-2\alpha\} $ and $ 1<q_0 < 1/(1- \alpha r) $, then
    \begin{small}
    \begin{equation}
      \label{eq:u-to-u-alpha}
      \begin{aligned}
        & \nm{u}_{
          W^{\min\{2\alpha+\beta,1\},q_0}(0,T;L^2(\Omega))
        } \\
        \leqslant{} & C_{\alpha,\beta,q_0,r,\nu,u_*,u^*,T,\Omega} \big(
          1 + \nm{y_0}_{\dot H^{2r}(\Omega)} +
          \nm{y_d}_{H^{1}(0,T;L^2(\Omega))} +
          \nm{u}_{W^{\beta,q_0}(0,T;L^2(\Omega))}
        \big).
      \end{aligned}
    \end{equation}
    \end{small}
  \end{lemma}
	\begin{proof}
    We only prove the case $ \beta <1-2\alpha $, since the other cases can be
    proved analogously. For simplicity, we denote by $ C $ in this proof a
    generic positive constant depending only on $ \alpha $, $ \beta $, $ q_0 $,
    $r$, $ \nu $, $ u_* $, $ u^* $, $ T $ and $ \Omega $, and its value may
    differ in different places. Some straightforward calculations give
		\[
		\nm{S\D_{0+}^\alpha y_0}_{W^{1,q_0}(0,T;L^2(\Omega))}
		\leqslant C \nm{y_0}_{\dot H^{2r}(\Omega)}
		\quad\text{(by \cref{eq:SDalpha_v_h1})}
		\]
		and
		\begin{align*}
		\nm{Su}_{{}_0W^{\alpha+\beta,q_0}(0,T;L^2(\Omega))}
		& \leqslant C \nm{u}_{{}_0W^{\beta,q_0}(0,T;L^2(\Omega))}
		\quad\text{(by \cref{eq:Sg-real})} \\
		& \leqslant C \nm{u}_{W^{\beta,q_0}(0,T;L^2(\Omega))}
		\quad\text{(by the fact $ \beta q_0 < 1 $),}
		\end{align*}
		so that
		\begin{align*}
		\nm{y}_{W^{\alpha+\beta,q_0}(0,T;L^2(\Omega))} &=
		\nm{S(u+\D_{0+}^\alpha y_0)}_{W^{\alpha+\beta,q_0}(0,T;L^2(\Omega))} \\
		& \leqslant C\big(
		\nm{u}_{W^{\beta,q_0}(0,T;L^2(\Omega))} +
		\nm{y_0}_{\dot H^{2r}(\Omega)}
		\big).
		\end{align*}
		By \cref{eq:S^*g-real} and the fact $(\alpha+\beta)q_0<1$,
		\begin{align*}
		& \nm{p}_{
			{}^0W^{2\alpha+\beta,q_0}(0,T;L^2(\Omega))
		} =
		\nm{S^*(y-y_d)}_{
			{}^0W^{2\alpha+\beta,q_0}(0,T;L^2(\Omega))
		} \\
		\leqslant{} & C \nm{y-y_d}_{{}^0W^{\alpha+\beta,q_0}(0,T;L^2(\Omega))}
		\leqslant C \nm{y-y_d}_{W^{\alpha+\beta,q_0}(0,T;L^2(\Omega))} \\
		\leqslant{} & C\big(
		\nm{y}_{W^{\alpha+\beta,q_0}(0,T;L^2(\Omega))} +
		\nm{y_d}_{H^1(0,T;L^2(\Omega))}
		\big).
		\end{align*}
		In addition,
		\begin{align*}
		\nm{u}_{W^{2\alpha+\beta,q_0}(0,T;L^2(\Omega))} &=
		\nm{f(p)}_{W^{2\alpha+\beta,q_0}(0,T;L^2(\Omega))}
		\quad\text{(by \cref{eq:u})} \\
		& \leqslant C \big(
		1 + \nm{p}_{
			W^{2\alpha+\beta,q_0}(0,T;L^2(\Omega))
		}
		\big) \quad\text{(by \cref{lem:fv}).}
		\end{align*}
		Finally, combining the above three estimates proves \cref{eq:u-to-u-alpha} and
		hence this lemma.
	\end{proof}

\begin{lemma} 
    \label{lem:u-to-u-p}
    Assume that $ 1 < q \leqslant 2 $, $ 0<A<\infty $, and $ u(t) = u_1(t) +
    u_2(t) $ for each $ 0 < t < T $, with $ u_1 \in {}_0W^{1,q}(0,T;L^2(\Omega))
    $ and $ u_2 \in W_\text{loc}^{1,\infty}(0,T;L^2(\Omega)) $. If
    \begin{equation*}
      \nm{u_2'(t)}_{L^2(\Omega)} \leqslant
      A \big( t^{\alpha r +\alpha -1}+ (T-t)^{\alpha -1} \big) ,
      \quad\text{a.e.}~ 0 < t < T,
    \end{equation*}
    then there exists a decomposition
    \begin{equation}
      \label{eq:u1+u2=u}
      u(t) = \widetilde u_1(t) + \widetilde u_2(t),
      \quad 0 < t < T,
    \end{equation}
    such that
    \begin{equation}
      \label{eq:wtu1}
      \nm{\widetilde u_1}_{{}_0W^{1,q+\alpha/(1-\alpha)}(0,T;L^2(\Omega))}
      \leqslant C\big(
        \nm{u_1}_{{}_0W^{1,q}(0,T;L^2(\Omega))} +
        \nm{y_d}_{H^{1}(0,T;L^2(\Omega))}
      \big)
    \end{equation}
    and
    \begin{small}
    \begin{equation}
      \label{eq:wtu2}
      \begin{aligned}
        \nm{\widetilde u_2'(t)}_{L^2(\Omega)} & \leqslant
        C \Big(
          A + \nm{u(0)}_{L^2(\Omega)} +
          \nm{y_0}_{\dot H^{2r}(\Omega)} +
          \nm{y_d}_{H^1(0,T;L^2(\Omega))} \\
          & \qquad\qquad\qquad {} + \nm{u_1}_{{}_0W^{1,q}(0,T;L^2(\Omega))}
        \Big) \big( t^{\alpha r +\alpha-1} + (T-t)^{\alpha -1} \big)
      \end{aligned}
    \end{equation}
    \end{small}
    for almost all $ 0 < t < T $,
    where $ C $ is a positive constant depending
    only on $ \alpha $, $r$, $ \nu $, $ u_* $, $ u^* $, $ q $, $ T $ and $ \Omega $.
\end{lemma}
\begin{proof}
  A simple calculation gives, by \cref{eq:y-p-def}, that
  \begin{equation}
    \label{eq:P=I1+I2}
    p = \mathbb I_1 + \mathbb I_2,
  \end{equation}
  where
  \begin{align*}
    \mathbb I_1 &:= S^*\big( Su_1 - y_d - (Su_1 - y_d)(T) \big), \\
    \mathbb I_2 &:= S^*\big(
      S(u_2 + \D_{0+}^\alpha y_0) + (Su_1 - y_d)(T)
    \big).
  \end{align*}
  We have
  \begin{align*}
    & \nm{\mathbb I_1}_{{}^0W^{1+\alpha,q}(0,T;L^2(\Omega))} \\
    \leqslant{} &
    C \nm{ Su_1 - y_d - (Su_1 - y_d)(T) }_{
      {}^0W^{1,q}(0,T;L^2(\Omega))
    } \quad\text{(by \cref{eq:S^*g-real})} \\
    \leqslant{} &
    C \big(
      \nm{Su_1}_{{}_0W^{1,q}(0,T;L^2(\Omega))} +
      \nm{y_d}_{H^{1}(0,T;L^2(\Omega))}
    \big) \\
    \leqslant{} &
    C \big(
      \nm{u_1}_{{}_0W^{1,q}(0,T;L^2(\Omega))} +
      \nm{y_d}_{H^{1}(0,T;L^2(\Omega))}
    \big) \quad\text{(by \cref{eq:Sg-real}).}
  \end{align*}
  As $ \scriptstyle {}^0W^{1+\alpha,q}(0,T;L^2(\Omega)) $ is continuously
  embedded into $ \scriptstyle {}^0W^{1,q+\alpha/(1-\alpha)}(0,T;L^2(\Omega)) $,
  it holds
  \begin{equation}
    \label{eq:E1}
    \nm{\mathbb I_1}_{{}^0W^{1,q+\alpha/(1-\alpha)}(0,T;L^2(\Omega))}
    \leqslant C \big(
      \nm{u_1}_{{}_0W^{1,q}(0,T;L^2(\Omega))} +
      \nm{y_d}_{H^{1}(0,T;L^2(\Omega))}
    \big).
  \end{equation}
  From \cref{eq:90} it follows that
  \begin{align*}
    \nm{(Su_2)'(t)}_{L^2(\Omega)}\!
    \leqslant C(A+u_2(0))\big( t^{\alpha-1}+ \omega_2(T-t) \big).
  \end{align*}
  Therefore,  the fact that
  \begin{small}
  \begin{align*}
    \nm{(Su_1)(T)}_{L^2(\Omega)} \!\leqslant\!
    C \nm{Su_1}_{{}_0W^{1+\alpha/2,q}(0,T;L^2(\Omega))}
    \!\leqslant\! C \nm{u_1}_{{}_0W^{1\!-\!\alpha/2,q}(0,T;L^2(\Omega))}
    \ \ \text{(by \cref{eq:Sg-real})}
  \end{align*}
  \end{small}
  and
  \[
    \nm{(S \D_{0+}^\alpha y_0)' (t)} \leqslant Ct^{\alpha r-1}\nm{y_0}_{\dot{H}^{2r}(\Omega)} \quad \text{(by \cref{eq:SDalpha_v_h1})},
  \]
  together with   \cref{lem:SgS*g} and
  \cref{eq:(S^*v)'-growth}, yields
  \begin{small}
  \begin{equation}
    \label{eq:E2}
    \begin{aligned}
      \nm{\mathbb I_2'(t)}_{L^2(\Omega)} & \leqslant
      C \Big(
        A + \nm{u(0)}_{L^2(\Omega)} +
        \nm{y_0}_{\dot H^{2r}(\Omega)} +
        \nm{u_1}_{{}_0W^{1,q}(0,T;L^2(\Omega))} \\
        & \qquad\qquad\qquad\quad {} +
        \nm{y_d}_{H^{1}(0,T;L^2(\Omega))}
      \Big) \big(
        t^{\alpha r +\alpha -1}+ (T-t)^{\alpha-1}
      \big)
    \end{aligned}
  \end{equation}
  \end{small}
  for all $ 0 < t < T $.

  Finally, letting
  \begin{align}
    \widetilde u_1(t) &:= \int_0^t
    f'( p(s) )
    \mathbb I_1'(s) \, \mathrm{d}s,
    \quad 0 < t < T, \label{eq:wtu1-def} \\
    \widetilde u_2(t) &:= u(0) + \int_0^t
    f'( p(s) )
    \mathbb I_2'(s) \, \mathrm{d}s,
    \quad 0 < t < T, \label{eq:wtu2-def}
  \end{align}
  by \cref{eq:u,eq:P=I1+I2} we obtain
  \begin{align*}
    \widetilde u_1(t) \!+\! \widetilde u_2(t) & =
    u(0) +f(p(t))- f(p(0))
    = u(t), \quad 0 < t \leqslant T,
  \end{align*}
  which proves \cref{eq:u1+u2=u}. Furthermore, \cref{eq:wtu1} follows from
  \cref{eq:E1,eq:wtu1-def}, and \cref{eq:wtu2} follows from
  \cref{eq:E2,eq:wtu2-def}. This completes the proof.
\end{proof}

Finally, we are in a position to prove \cref{thm:regu-u,thm:regu-y-p}.

\medskip\noindent{\bf Proof of \cref{thm:regu-u}}. By the fact $ u\in
  U_{ad}\subset L^2(0,T;L^2(\Omega)) $, a similar proof as that of
  \cref{lem:u-to-u} gives
  \[
    \nm{u}_{W^{\alpha,q_0}(0,T;L^2(\Omega))}
    \leqslant C,
  \]
  and then applying \cref{lem:u-to-u} several times yields
  \[
    \nm{u}_{W^{1,q_0}(0,T;L^2(\Omega))} \leqslant C,
  \]
  where $ q_0 := \frac12(1+1/(1-\alpha r)) < 1/(1-\alpha r) $
  and $ C $ is a positive constant depending only on $ \alpha $, $r$, $ \nu $,
  $ u_* $, $ u^* $, $ y_0 $, $ y_d $, $ T $ and $ \Omega $. Letting
  \[
    m := \min\{n \in \mathbb N:\ q_0+n\alpha/(1-\alpha) \geqslant 2 \}
  \]
  and applying \cref{lem:u-to-u-p} $ m $ time(s) then prove \cref{thm:regu-u} ($u_2=0$ for the first time).
\hfill\ensuremath{\blacksquare}

\medskip\noindent{\bf Proof of \cref{thm:regu-y-p}}. Let $ y_1 := S u_1 $ and
  $ y_2 := S(u_2 + \D_{0+}^\alpha y_0) $, where $ u_1 $ and $ u_2 $ are
  defined in \cref{thm:regu-u}. By \cref{eq:Sg-real-complex,eq:u1} we have
  \begin{equation}
    \nm{y_1}_{{}_0H^{1+\alpha}(0,T;L^2(\Omega))} +
    \nm{y_1}_{{}_0H^1(0,T;\dot H^2(\Omega))}
    \leqslant C,
  \end{equation}
  and by \cref{eq:u2}, \cref{eq:90}, \cref{eq:91} and \cref{eq:SDalpha_v_h1}
  we have that $$ y_2 \in C([0,T];L^2(\Omega)) \bigcap C^1((0,T);\dot
  H^1(\Omega)) $$ and
  \begin{align*}
    &  \nm{y_2'(t)}_{L^2(\Omega)} \leqslant C
    \big( t^{\alpha r-1} + \omega_2(T-t) \big),
    \quad 0 < t < T, \\
    &\nm{y_2'(t)}_{\dot H^1(\Omega)} \leqslant C
    \big( t^{(r-1/2)\alpha-1} + \omega_1(T-t) \big),
    \quad 0 < t < T,
  \end{align*}
  where $ C $ is a positive constant depending only on $ \alpha $, $r$, $ \nu $, $
  u_* $, $ u^* $, $ y_0 $, $ y_d $, $ T $ and $ \Omega $.
  Hence, $ y = y_1 + y_2
  $ is the desired decomposition in \cref{thm:regu-y-p}. Since the rest of this
  theorem can be proved analogously, this concludes the proof.
\hfill\ensuremath{\blacksquare}

\section{Discretization}
\label{sec:conv}
Assume that $0<r<1$, $ M >1 $ is an integer and
\begin{equation}
  \label{eq:sigma}
  \left\{
    \begin{array}{ll} \sigma_1 >\max\left\{
        1, \frac{2-\alpha}{(2r-1)\alpha+1}
      \right\}, \\
      \sigma_2> \max\left\{
        1, \frac{2-\alpha}{\alpha+1}
      \right\}.
    \end{array}
  \right.
\end{equation}
Let $ 0 = t_0 < t_1 < \ldots < t_{2M} = T$ be a graded partition of the temporal
interval $[0, T]$ with
\begin{align*}
  t_j  :=\left\{\begin{array}{ll}
  \big( \frac jM \big)^{\sigma_{1}} \frac{T}{2},& \text{ if }
  \ 0 \leqslant j \leqslant M, \\\\
  T - \big( 2-\frac jM \big)^{\sigma_{2}} \frac{T}{2}, & \text{ if }
  \ M < j \leqslant 2M.
\end{array}
  \right.
\end{align*}
Define $\tau_j := t_j-t_{j-1}$ for each $ 1 \leqslant j \leqslant 2M $ and set $
\tau := \max\{\tau_j: 1 \leqslant j \leqslant 2M\}.$ Let $ \mathcal K_h $ be a
conventional conforming and shape-regular triangulation of $ \Omega $ consisting
of $ d $-simplexes with mesh size $h:=\max_{K \in \mathcal K_h} \{
\text{diameter of } K\}$. Then we introduce the following finite element spaces:
\begin{align*}
  \mathcal V_h &:= \left\{
    v_h \in \dot H^1(\Omega):\
    \text{$ v_h $ is linear on $ K $ for each $ K \in \mathcal K_h $}
  \right\}, \\
  \mathcal W_{h\tau} &:= \left\{
    V \in L^2(0,T;\mathcal V_h):\,
    V \ \text{is constant on $ (t_{j-1},t_j) $ for each }
    1 \leqslant j \leqslant 2M
  \right\}.
\end{align*}
For any $ g \in L^2(0,T;L^2(\Omega)) $, define $ S_{h\tau}g$, $S^*_{h\tau}g \in
\mathcal W_{h\tau}$ respectively by that
\begin{align}
  \big(
    \D_{0+}^{\alpha/2} S_{h\tau} g ,\D_{T-}^{\alpha/2} V
  \big)_{\Omega \times (0,T)} +\big(
    \nabla S_{h\tau} g , \nabla  V
  \big)_{\Omega \times (0,T)} =
  \dual{g,V}_{\Omega \times (0,T)},  \label{def:shtau} \\
  \big(
    \D_{T-}^{\alpha/2} S^*_{h\tau} g ,\D_{0+}^{\alpha/2} V
  \big)_{\Omega \times (0,T)} + \big(
    \nabla S^*_{h\tau} g , \nabla  V
  \big)_{\Omega \times (0,T)} =
  \dual{g,V}_{\Omega \times (0,T)}, \label{def:shtau*}
\end{align}
for all $ V \in \mathcal W_{h\tau} $. It is evident that
\begin{equation}
  \label{eq:Shtau-Shtau*}
  (S_{h \tau} g_1, g_2)_{\Omega \times (0,T)} =
  (g_1, S_{h \tau}^*g_2)_{\Omega \times (0,T)}
\end{equation}
for all $ g_1, g_2 \in \mathcal W_{h\tau} $.
\begin{remark}
  By \cref{lem:dual} and the Lax-Milgram theorem, it is easy to verify that the
  operators $ S_{h\tau}$ and $S^*_{h\tau}$ are well defined.
\end{remark}
With the above two operators, we consider the following optimal control problem:
\begin{equation}
  \label{eq:model2}
  \min\limits_{
    \substack{
      U \in U_{\text{ad}}
    }
  } J(U) = \frac12 \nm{S_{h\tau} (U+\D_{0+}^{\alpha}y_0) - y_d}_{L^2(0,T;L^2(\Omega))}^2 +
  \frac\nu2 \nm{U}_{L^2(0,T;L^2(\Omega))}^2.
\end{equation}
Similar to the continuous case, there exists a unique discrete control $ U \in
U_\text{ad} $ such that
\begin{equation}
  \label{eq:optim_cond_full}
  \left(
    S^*_{h\tau}  \big(
      S_{h\tau} (U+\D_{0+}^{\alpha} y_0) - y_d
    \big)+ \nu U, \, v-U
  \right)_{\Omega \times (0,T)} \geqslant 0, \quad  \forall v\in U_{ad}.
\end{equation}
The corresponding discrete state $Y$ and co-state $P$ are defined respectively
by
\begin{equation}
  \label{eq:Y-P}
  Y:=S_{h\tau}(U+\D_{0+}^{\alpha} y_0)
  \quad \text{and} \quad P:=S_{h\tau}^*(Y-y_d).
\end{equation}
\begin{remark}
  Since the co-state $P$ is in the finite dimensional space $\mathcal
  W_{h\tau}$, the control $u$ is indirectly discretized by the projection
  $$U=Q_{U_{ad}}(-P/\nu),$$ where $Q_{U_{ad}}$ is the $L^2$ projection onto the
  admissible set $U_{ad}$. This is the key point of the variational
  discretization concept (cf. \cite{M2008}).
\end{remark}

In this section, we use $\bar a \lesssim \bar b$ to denote $\bar a \leqslant C
\bar b$, where $C$ is a positive constant independent of $h$ and $M$. The main
result of this section is the following theorem.
\begin{theorem}
  \label{th:ape}
  Assume that $ 0 < r < \min\{1/2, (1-\alpha)/\alpha\} $. Let $u$ and $y$ be the
  control and state of \cref{eq:model} respectively, and let $U$ and $Y$ be the
  control and state of \cref{eq:model2} respectively. If $y_0 \in \dot
  H^{2r}(\Omega) $ and $y_d \in H^{1}(0,T;L^2(\Omega)) $, then
  \begin{equation}
    \label{eq:main-error-est}
    \nm{u-U}_{L^2(0,T;L^2(\Omega))} +
    \nm{y-Y}_{L^2(0,T;L^2(\Omega))}
    \lesssim h^{\min\{1/\alpha+2r,2\}} + M^{-1}.
  \end{equation}
\end{theorem}
\begin{remark}
  By following a similar routine in the proof of \cref{th:ape}, we can show that
  the estimate \eqref{eq:main-error-est} also holds for $0<r<1$ with $r\neq
  \frac{1}{\alpha}-1, \frac{1}{\alpha}-\frac12.$
\end{remark}
\begin{remark}
  Note that, under the condition that $ y_0 = 0 $ and $ y_d \in
  H^1(0,T;L^2(\Omega)) $, Jin et
  al.~\cite[Theorem~3.10]{jin2017pointwise-in-time} derived temporal accuracy $
  O(\tau^{1/2+\min\{1/2,\alpha-\epsilon\}}) $, where $ \epsilon > 0 $ can be
  arbitrarily small.
\end{remark}
\subsection{Proof of \texorpdfstring{\cref{th:ape}}{}} 
Throughout this subsection, $ u $, $ y $ and $p$ are the control, state and
co-state of \cref{eq:model}, respectively. By \cref{eq:optim_cond},
\cref{eq:optim_cond_full} and the standard technique in
\cite{Hinze2009Variational}, we obtain
\begin{align*}
  & \|u-U\|_{L^2(0,T;L^2(\Omega))}+ \nm{y-Y}_{L^2(0,T;L^2(\Omega))} \\
  \lesssim{}&
  \nm{y-\widetilde{Y}}_{L^2(0,T;L^2(\Omega))} +
  \nm{p-\widetilde{P}}_{L^2(0,T;L^2(\Omega))},
\end{align*}
where
\begin{align}
  \widetilde Y &:= S_{h\tau}(u+\D_{0+}^\alpha y_0), \\
  \widetilde P &:= S_{h\tau}^*(y-y_d).
\end{align}
Therefore, to conclude the proof of \cref{th:ape}, it remains to prove
\begin{align}
  \nm{y-y_h}_{L^2(0,T;L^2(\Omega))} +
  \nm{p-p_h}_{L^2(0,T;L^2(\Omega))}
  & \lesssim h^{\min\{1/\alpha+2r,2\}},
  \label{eq:y-yh-p-ph} \\
  \nm{y_h-\widetilde Y}_{L^2(0,T;L^2(\Omega))} +
  \nm{p_h-\widetilde P}_{L^2(0,T;L^2(\Omega))}
  & \lesssim M^{-1}, \label{eq:yh-wtY-ph-wtP}
\end{align}
where $ y_h $ is the solution of the equation
\begin{equation}
  \begin{cases}
    \D_{0+}^\alpha(y_h - Q_h y_0) - \Delta_h y_h = Q_hu, \\
    y_h(0) = Q_hy_0,
  \end{cases}
\end{equation}
and $ p_h $ is the solution of the equation
\begin{equation}
  \begin{cases}
    (\D_{T-}^\alpha - \Delta_h) p_h = Q_h(y-y_d), \\
    p_h(T) = 0.
  \end{cases}
\end{equation}
Here $ Q_h $ is the $ L^2(\Omega) $-orthogonal projection operator onto $
\mathcal V_h $ and $ \Delta_h: \mathcal V_h \to \mathcal V_h $ is the discrete
Laplace operator defined by
\begin{equation}
  (\Delta_h v_h, w_h)_\Omega =
  -(\nabla v_h, \nabla w_h)_\Omega
  \quad \forall v_h, w_h \in \mathcal V_h.
\end{equation}


Let us first prove estimate \cref{eq:y-yh-p-ph} in the following lemma.
\begin{lemma}
  \label{lem:y-yh}
  Under the condition of \cref{th:ape}, we have
  \begin{align}
    \nm{y-y_h}_{L^2(0,T;L^2(\Omega))}
    & \lesssim h^{\min\{1/\alpha+2r,2\}},
    \label{eq:y-yh} \\
    \nm{p-p_h}_{L^2(0,T;L^2(\Omega))}
    & \lesssim h^2. \label{eq:p-ph}
  \end{align}
\end{lemma}
\begin{proof} 
  By \cref{eq:Sg-real-complex} and the fact $ u \in U_\text{ad} $, we have
  \[
    Su \in {}_0H^\alpha(0,T;L^2(\Omega))
    \cap L^2(0,T;\dot H^2(\Omega)),
  \]
  so that by interpolation we obtain
  \[
    Su \in {}_0H^{\alpha/2}(0,T;\dot H^1(\Omega)).
  \]
  In addition, \cref{lem:regu2} implies
  \[
    S\D_{0+}^\alpha y_0 \in
    {}_0H^{\alpha/2}\big(
      0,T;\dot H^{\min\{1/\alpha+2r-1,1\}}(\Omega)
    \big) \cap L^2\big(
      0,T;\dot H^{\min\{1/\alpha+2r,2\}}(\Omega)
    \big).
  \]
  Consequently, we conclude from the fact $ y = S(u + \D_{0+}^\alpha y_0) $ that
  \[
    y \in {}_0H^{\alpha/2}\big(
      0,T;\dot H^{\min\{1/\alpha+2r-1,1\}}(\Omega)
    \big) \cap L^2\big(
      0,T;\dot H^{\min\{1/\alpha+2r,2\}}(\Omega)
    \big).
  \]
  A routine energy argument (cf.~\cite{Li2019SIAM}) then yields \cref{eq:y-yh}.
  Since \cref{eq:p-ph} can be derived analogously, this completes the proof.
\end{proof}

Then let us prove estimate \cref{eq:yh-wtY-ph-wtP}. Similar to the properties of
$ y $ and $ p $ presented in \cref{thm:regu-y-p}, under the condition of
\cref{th:ape}, we have the following properties: there exist decompositions
\begin{equation}
  \label{eq:yh-ph}
  y_h = y_{h,1} + y_{h,2}, \quad
  p_h = p_{h,1} + p_{h,2},
\end{equation}
where
\begin{align}
  & \nm{y_{h,1}}_{{}_0H^{1+\alpha}(0,T;L^2(\Omega))} +
  \nm{\Delta_h y_{h,1}}_{{}_0H^1(0,T;L^2(\Omega))}
  \lesssim 1, \label{eq:regu-yh1} \\
  & \nm{p_{h,1}}_{{}^0H^{1+\alpha}(0,T;L^2(\Omega))} +
  \nm{\Delta_h p_{h,1}}_{{}^0H^1(0,T;L^2(\Omega))}
  \lesssim 1, \label{eq:regu-ph1} \\
  & \nm{y_{h,2}'(t)}_{L^2(\Omega)}
  \lesssim t^{\alpha r - 1} + \omega_2(T-t)
  \quad \forall 0 < t < T, \\
  & \nm{y_{h,2}'(t)}_{\dot H^1(\Omega)} \lesssim
  t^{\alpha r - \alpha/2 - 1} + \omega_1(T-t)
  \quad \forall 0 < t < T, \\
  & \nm{p_{h,2}'(t)}_{L^2(\Omega)}
  \lesssim t^{\alpha r + \alpha - 1} + (T-t)^{\alpha-1}
  \quad \forall 0 < t < T, \\
  & \nm{p_{h,2}'(t)}_{\dot H^1(\Omega)} \lesssim
  t^{\alpha r + \alpha/2 - 1} + (T-t)^{\alpha/2-1}
  \quad \forall 0 < t < T.
\end{align}
Here $ \omega_1 $ and $ \omega_2 $ are defined by \cref{eq:omega1,eq:omega2},
respectively. For any $ v \in L^2(0,T;L^2(\Omega))$, space, define $ Q_\tau v
\in L^\infty(0,T;L^2(\Omega)) $ by
\[
  (Q_\tau v)|_{(t_{j-1},t_j)} :=
  \frac1{\tau_j} \int_{t_{j-1}}^{t_j} v(t) \, \mathrm{d}t,
  \quad  \forall  1  \leqslant j \leqslant 2M.
\]
By \cref{eq:regu-yh1,eq:regu-ph1} it is standard that
\begin{align}
  \nm{(I-Q_\tau)y_{h,1}}_W \lesssim M^{-1+\alpha/2}, \label{eq:I-Q-yh1} \\
  \nm{(I-Q_\tau)p_{h,1}}_W \lesssim M^{-1+\alpha/2}, \label{eq:I-Q-ph1}
\end{align}
where
\[
  \nm{\cdot}_W:= \nm{\cdot}_{{}_0H^{\alpha/2}(0,T;L^2(\Omega))}+
  \nm{\cdot}_{L^2(0,T;\dot{H}^1(\Omega))}.
\]

\begin{lemma}
  \label{lem:I-Qtau}
  Under the condition of \cref{th:ape}, we have
  \begin{align}
    \nm{(I-Q_\tau)y_h}_{W} & \lesssim M^{-1+\alpha/2},
    \label{eq:y-11} \\
    \nm{(I-Q_\tau)p_h}_{W} & \lesssim M^{-1+\alpha/2}.
    \label{eq:p-11}
  \end{align}
\end{lemma}
\begin{proof}
  We only prove \cref{eq:p-11}, since \cref{eq:y-11} can be derived analogously.
  By \cref{eq:yh-ph,eq:I-Q-ph1} it remains to prove
  \begin{align}
    \nm{(I-Q_\tau)p_{h,2}}_{{}_0H^{\alpha/2}(0,T;L^2(\Omega))}
    & \lesssim M^{-1+\alpha/2}, \label{eq:p2} \\
    \nm{(I-Q_\tau)p_{h,2}}_{L^2(0,T;\dot H^1(\Omega))}
    & \lesssim M^{-1+\alpha/2}. \label{eq:p2-h1}
  \end{align}
  We will give the proof of \cref{eq:p2}, the proof of \cref{eq:p2-h1} being
  easier. To this end, we proceed as follows. Let
  \begin{align*}
    g_1(t) := \begin{cases}
      (p_{h,2} - Q_\tau p_{h,2})(t) & \text{ if } t \in (0,T/2), \\
      0 & \text{ if } t \in (-\infty,0) \cup (T/2, \infty),
    \end{cases} \\
    g_2(t) := \begin{cases}
      (p_{h,2}-Q_\tau p_{h,2})(t) & \text{ if } t \in (T/2,T), \\
      0 & \text{ if } t \in (-\infty,T/2) \cup (T, \infty).
    \end{cases}
  \end{align*}
  By \cite[Lemma 16.3]{Tartar2007} we obtain
  \begin{equation}
    \label{eq:g1}
    \nm{g_1}^2_{H^{\alpha/2}(-\infty,\infty;L^2(\Omega))}
    \lesssim \mathbb{I}_1 + \mathbb{I}_2 + \mathbb{I}_3 + \mathbb{I}_4,
  \end{equation}
  where
  \begin{align*}
    \mathbb I_1 &:=\int_{0}^{t_1} \int_{0}^{t_1}
    \frac{\nm{g_1(s)-g_1(t)}_{L^2(\Omega)}^2}{\snm{s- t}^{1+\alpha}}
    \, \mathrm{d}s \, \mathrm{d}t, \\
    \mathbb I_2 &:= \sum_{j=2}^{M} \int_{t_{j-1}}^{t_j} \int_{t_{j-1}}^{t_j}
    \frac{\nm{g_1(s)- g_1(t)}_{L^2(\Omega)}^2}{\snm{s-t}^{1+\alpha}}
    \, \mathrm{d}s \, \mathrm{d}t, \\
    \mathbb I_3 &:= \int_{0}^{t_1} \nm{g_1(t)}_{L^2(\Omega)}^2
    \big(
      (t_1-t)^{-\alpha} + t^{-\alpha}
    \big) \, \mathrm{d}t, \\
    \mathbb I_4 &:= \sum_{j=2}^{M}
    \int_{t_{j-1}}^{t_j} \nm{g_1(t)}_{L^2(\Omega)}^2
    \big(
      (t_j-t)^{-\alpha} + (t-t_{j-1})^{-\alpha}
    \big) \, \mathrm{d}t.
  \end{align*}
  By \cref{eq:regu-p2}, an elementary calculation gives the following four
  estimates:
  \begin{align*} 
    \mathbb I_1 &=
    2 \int_{0}^{t_1} \int_{t}^{t_1}
    \frac{\nm{g_1(s)-g_1(t)}_{L^2(\Omega)}^2}{\snm{s-t}^{1+\alpha}}
    \, \mathrm{d}s \, \mathrm{d}t \phantom{++++++++++++++++}  \\
    & \lesssim
    \int_{0}^{t_1} \int_{t}^{t_1}
    \frac{\left(s^{\alpha r +\alpha }-t^{\alpha r +\alpha }\right)^2}{(s-t)^{1+\alpha}}
    \, \mathrm{d}s \, \mathrm{d}t \\
    &=
    \int_{0}^{t_1} \int_{t}^{t_1} \left(
      \left(\frac{s}{t}\right)^{\alpha r +\alpha } -1
    \right)^2 \left(\frac{s-t}{t} \right)^{-(1+\alpha)}
    t^{2\alpha r+\alpha-1} \, \mathrm{d}s \, \mathrm{d}t \\
    & \lesssim
    \int_{0}^{t_1} t^{2\alpha r+\alpha } \int_{1}^{t_1/t}
    \big(x^{\alpha r +\alpha}-1 \big)^2 \big(x-1\big)^{-(1+\alpha)}
    \, \mathrm{d}x \, \mathrm{d}t \\
    & \lesssim
    \int_{0}^{t_1} t^{2\alpha r+ \alpha }
    \left(1+ \left(\frac{t_1}{t} \right)^{2\alpha r +\alpha}\right)\,  \mathrm{d}t \\
    & \lesssim t_1^{2\alpha r + \alpha + 1},
  \end{align*}
  \begin{align*}
    \mathbb I_2 &=
    2\sum_{j=2}^M \int_{t_{j-1}}^{t_j}
    \int_t^{t_j} \frac{\nm{g_1(s)-g_1(t)}_{L^2(\Omega)}^2}{
      (s-t)^{1+\alpha}
    } \, \mathrm{d}s \, \mathrm{d}t \phantom{+++++++++++++} \\
    & \lesssim
    \sum_{j=2}^M \int_{t_{j-1}}^{t_j}
    \int_t^{t_j} \frac{
      (s-t)^2 \nm{p_{h,2}'}_{L^\infty(s,t;L^2(\Omega))}^2
    }{
      (s-t)^{1+\alpha}
    } \, \mathrm{d}s \, \mathrm{d}t \\
    & \lesssim
    \sum_{j=2}^M \big(
      1+t_{j-1}^{2(\alpha r+\alpha-1)}
    \big) \int_{t_{j-1}}^{t_j}
    \int_t^{t_j} (s-t)^{1-\alpha} \, \mathrm{d}s \, \mathrm{d}t \\
    & \lesssim
    \sum_{j=2}^M \big( 1+t_{j-1}^{2(\alpha r+\alpha-1)} \big)
    (t_j-t_{j-1})^{3-\alpha},
  \end{align*}
  \begin{align*}
    \mathbb I_3
    & \lesssim \nm{g_1}_{L^\infty(0,t_1;L^2(\Omega))}^2
    \int_0^{t_1} \big( (t_1-t)^{-\alpha} + t^{-\alpha }\big) \, \mathrm{d}t
    \phantom{++++++++++++} \\
    & \lesssim t_1^{2(\alpha r + \alpha)}
    \int_0^{t_1} \big( (t_1-t)^{-\alpha} + t^{-\alpha} \big) \, \mathrm{d}t \\
    & \lesssim t_1^{2\alpha r + \alpha + 1}
  \end{align*}
  and
  \begin{align*}
    \mathbb I_4 & \lesssim
    \sum_{j=2}^M (t_j-t_{j-1})^2
    \nm{p_{h,2}'}_{L^\infty(t_{j-1},t_j;L^2(\Omega))}^2
    \int_{t_{j-1}}^{t_j} \big(
      (t_j-t)^{-\alpha} + (t-t_{j-1})^{-\alpha}
    \big) \, \mathrm{d}t \\
    & \lesssim \sum_{j=2}^M \big(1 + t_{j-1}^{2(\alpha r + \alpha - 1)} \big)
    (t_j-t_{j-1})^{3-\alpha}.
  \end{align*}
  Since
  \begin{align*}
    \sum_{j=2}^M (t_j-t_{j-1})^{3-\alpha}
    &\lesssim M^{(\alpha-3)\sigma_1}
    \sum_{j=2}^M (j^{\sigma_1} - (j-1)^{\sigma_1})^{3-\alpha} \\
    & \lesssim
    M^{(\alpha-3)\sigma_1}
    \sum_{j=2}^M j^{(\sigma_1-1)(3-\alpha)} \\
    & \lesssim M^{\alpha-2}
  \end{align*}
  and
  \begin{align*} 
    & \sum_{j=2}^M t_{j-1}^{2(\alpha r + \alpha - 1)}
    (t_j-t_{j-1})^{3-\alpha} \\
    \lesssim{} &
    M^{-(2\alpha r + \alpha + 1)\sigma_1}
    \sum_{j=2}^M (j-1)^{2(\alpha r+\alpha-1)\sigma_1}
    \big(
      j^{\sigma_1} - (j-1)^{\sigma_1}
    \big)^{3-\alpha} \\
    \lesssim{} &
    M^{-(2\alpha r + \alpha + 1)\sigma_1}
    \sum_{j=2}^M j^{2(\alpha r+\alpha-1)\sigma_1}
    j^{(\sigma_1-1)(3-\alpha)} \\
    ={} &
    M^{-(2\alpha r + \alpha + 1)\sigma_1}
    \sum_{j=2}^M j^{(2\alpha r+\alpha+1)\sigma_1 + \alpha - 3} \\
    \lesssim{} &
    M^{\alpha-2} \quad \text{(by \cref{eq:sigma}),}
  \end{align*}
  combining the above estimates for $ \mathbb I_2 $ and $ \mathbb I_4 $ yields
  \[
    \mathbb I_2 + \mathbb I_4 \lesssim M^{\alpha-2}.
  \]
  In addition, combining the above estimate for $ \mathbb I_1 $ and $ \mathbb
  I_3 $ yields
  \[
    \mathbb I_1 + \mathbb I_3 \lesssim
    t_1^{2\alpha r + \alpha + 1}
    \lesssim M^{-(2\alpha r + \alpha + 1)\sigma_1}
    \lesssim M^{\alpha-2} \quad\text{(by \cref{eq:sigma}).}
  \]
  Consequently, we conclude from \cref{eq:g1} that
  \[
    \nm{g_1}_{H^{\alpha/2}(-\infty,\infty;L^2(\Omega))}
    \lesssim M^{\alpha-2}.
  \]
  A similar argument yields
  \[
    \nm{g_2}_{H^{\alpha/2}(-\infty,\infty;L^2(\Omega))}
    \lesssim M^{\alpha-2}.
  \]
  Therefore, \cref{eq:p2} follows from the estimate
  \begin{align*}
    \nm{g_1 + g_2}_{{}_0H^{\alpha/2}(0,T;L^2(\Omega))}
    & \lesssim \nm{g_1+g_2}_{H^{\alpha/2}(-\infty,\infty;L^2(\Omega))} \\
    & \lesssim \nm{g_1}_{H^{\alpha/2}(-\infty,\infty;L^2(\Omega))} +
    \nm{g_2}_{H^{\alpha/2}(-\infty,\infty;L^2(\Omega))}.
  \end{align*}
  This completes the proof.
\end{proof}

\begin{lemma}
  \label{lem:yh-wtY}
  Under the condition of \cref{th:ape}, we have
  \begin{align}
    \nm{y_h - \widetilde Y}_{L^2(0,T;L^2(\Omega))}
    \lesssim M^{-1}, \label{eq:yh-wtY} \\
    \nm{p_h - \widetilde P}_{L^2(0,T;L^2(\Omega))}
    \lesssim M^{-1}. \label{eq:ph-wtP}
  \end{align}
\end{lemma}
\begin{proof}
  By an energy argument (cf.~\cite{Li2019SIAM}), we obtain 
  \begin{align*}
    \nm{y_h - \widetilde Y}_{L^2(0,T;L^2(\Omega))}
    \lesssim M^{-\alpha/2}
    \nm{(I-Q_\tau)y_h}_W, \\
    \nm{p_h - \widetilde P}_{L^2(0,T;L^2(\Omega))}
    \lesssim M^{-\alpha/2}
    \nm{(I-Q_\tau)p_h}_W,
  \end{align*}
  so that \cref{eq:yh-wtY,eq:ph-wtP} follows from \cref{eq:y-11,eq:p-11},
  respectively. This completes the proof.
\end{proof}


\section{Numerical results}
\label{sec:numer}
This section provides three numerical experiments to verify the theoretical
results. We use uniform grids for the spatial
discretization and employ a fixed point method \cite{M2008} to solve the
discrete system. The convergent condition is that the difference of the discrete control (in $l^2$ norm) between two steps is less than 1e-13. We adopt the following setting:
\begin{align*}
  & \alpha =0.4 \ \text{or} \ 0.8; \quad r = 0 \ \text{or} \ 0.25;\\
  & \nu = 1; \quad T = 1; \quad \Omega = (0,1); \quad u_* = -0.1; \  u^* =0.1;\\
  & y_0(x) := x^{2r-0.49}(1-x) \quad \text{for all}
  \quad 0 \leqslant x \leqslant 1; \\
  & y_d(x,t) := x^{-0.49}(1-x)  \quad\text{for all}
  \quad \ 0 < x \leqslant 1 \text{ and } 0 \leqslant t \leqslant T.
\end{align*}
Let $ U^{m,n} $ be the numerical solution of \cref{eq:optim_cond_full} with the
mesh parameters $ M=2^m $, $ h= 1/n $ and 
\begin{equation}
\label{eq:sigma2}
\left\{
\begin{array}{ll} \sigma_1 = \max\left\{
1, \frac{2-\alpha}{(2r-1)\alpha+1}
\right\}, \\
\sigma_2= \max\left\{
1, \frac{2-\alpha}{\alpha+1}
\right\}.
\end{array}
\right.
\end{equation} Define the discrete state and co-state respectively as
  \begin{align*}
    Y^{m,n} :&=
    S_{h\tau}(U^{m,n}+ \D_{0+}^\alpha y_0), \\
    P^{m,n} :&=
    S^*_{h\tau}(S_{h\tau}(U^{m,n}+ \D_{0+}^\alpha y_0)-y_d).
  \end{align*}
  Throughout this section,
  $\nm{\cdot}_{L^2(0,T;L^2(\Omega))}$ is abbreviated to $\nm{\cdot}$ for
  convenience.


  \medskip\noindent\textbf{Experiment 1.} This experiment verifies the spatial
  accuracy.
  The reference solutions are $U^{14,512}$, $ Y^{14,512}$ and
  $P^{14,512}$. \cref{tab:space,tab:space2} demonstrate that the
  accuracies of state are close to $ \mathcal O(h^{\min\{2,1/\alpha +2r \}}) $, and this agrees well with
  \cref{th:ape}. In particular, it is observed that the convergence orders of the co-state and control are higher than the state in the case $\alpha = 0.8$.

  \begin{table}[H]
    \centering
    \caption{Convergence history with $r = 0$.}
    \label{tab:space}
    \footnotesize
    \setlength{\tabcolsep}{6pt}
    \begin{tabular}{cccccccc}
      \toprule &$n$&
      \multicolumn{2}{c}{$\nm{ Y^{14,512}- Y^{14,n}}$} &
      \multicolumn{2}{c}{$\nm{P^{14,512}- P^{14,n}}$} &
      \multicolumn{2}{c}{$\nm{U^{14,512}- U^{14,n}}$} \\
      \midrule
      \multirow{5}{*}{$\alpha=0.4$}
      &	$ 10 $ &  2.12e-3     & Order  & 1.60e-3 &Order  &  1.50e-3 &  Order \\
      &$ 20 $ & 5.94e-4  & 1.84 &   4.38e-4 & 1.87  &  4.16e-4    & 1.85   \\
      &$ 30 $ &2.78e-5   & 1.87  &  2.03e-4 &1.90 &   1.94e-4 &  1.89 \\
      &$ 40 $ & 1.61e-5  & 1.90    & 1.17e-4  &1.92  &  1.12e-4 & 1.91   \\
      &$ 50 $ & 1.05e-5   &1.99  &7.57e-5    &1.94& 7.26e-5  & 1.94    \\
      \midrule
      \multirow{5}{*}{$\alpha=0.8$}
      &$ 10 $ &  1.01e-2      & Order & 1.60e-3 &Order  & 1.49e-3 &  Order \\
      &$ 20 $ &4.23e-3  & 1.25 &  4.38e-4 & 1.87  &  4.13e-4    & 1.85   \\
      &$ 30 $ &2.53e-3   & 1.27   &  2.03e-4 &1.90&    1.93e-4 &  1.88 \\
      &$ 40$ & 1.74e-3  & 1.30 &   1.17e-4  &1.92  &   1.11e-4 & 1.90   \\
      &$ 50 $ &1.30e-3  &1.31  &7.57e-5    &1.94 &7.22e-5  & 1.94    \\
      \bottomrule
    \end{tabular}
  \end{table}

  \begin{table}[H]
    \centering
    \caption{Convergence history with $r = 0.25$.}
    \label{tab:space2}
    \footnotesize
    \setlength{\tabcolsep}{6pt}
    \begin{tabular}{cccccccc}
      \toprule &$n$&
      \multicolumn{2}{c}{$\nm{ Y^{14,512}- Y^{14,n}}$} &
      \multicolumn{2}{c}{$\nm{P^{14,512}- P^{14,n}}$} &
      \multicolumn{2}{c}{$\nm{U^{14,512}- U^{14,n}}$} \\
      \midrule
      \multirow{5}{*}{$\alpha=0.4$}
      &	$ 10 $ &  5.77e-4      & Order  & 1.56e-3 &Order  &  1.49e-3 &  Order \\
      &$ 20 $ & 1.45e-4  & 2.00 &   4.30e-4 & 1.86  &  4.14e-4    & 1.85   \\
      &$ 30 $ &6.43e-5   & 2.01  &  2.00e-4 &1.89 &   1.93e-4 &  1.89 \\
      &$ 40 $ & 3.58e-5  & 2.03    & 1.15e-4  &1.91  &  1.11e-4 & 1.90   \\
      &$ 50 $ & 2.28e-5   &2.03  &7.46e-5    &1.94& 7.23e-5  & 1.94     \\
      \midrule
      \multirow{5}{*}{$\alpha=0.8$}
      &$ 10 $ &  1.70e-3      & Order & 1.57e-3 &Order  & 1.48e-3 &  Order \\
      &$ 20 $ &5.07e-4  & 1.74  &  4.31e-3 & 1.86  &  4.12e-3    & 1.85   \\
      &$ 30 $ &2.46e-4   & 1.78   & 2.00e-3 &1.89 &   1.92e-4 &  1.88 \\
      &$ 40$ & 1.45e-4  & 1.83  &   1.15e-3  &1.91 &   1.11e-4 & 1.91   \\
      &$ 50 $ & 9.58e-5   &1.86  &7.47e-4    &1.94 &7.21e-5  & 1.93     \\
      \bottomrule
    \end{tabular}
  \end{table}

  \medskip\noindent\textbf{Experiment 2.} This experiment investigates the
  temporal accuracy with graded temporal grids. The reference solutions are $ U^{14,512}$, $Y^{14,512}$ and
  $ P^{14,512}$. \cref{tab:time3,tab:time4}
  illustrate that the temporal accuracy of the numerical control, state and
  co-state are close to $ \mathcal O(M^{-1}) $, which agrees well with
  \cref{th:ape}.

  \begin{table}[H]
    \centering
    \caption{Convergence history with $r=0$.}
    \label{tab:time3}
    \footnotesize
    \setlength{\tabcolsep}{6pt}
    \begin{tabular}{cccccccc}
      \toprule &$m$&
      \multicolumn{2}{c}{$\nm{ Y^{m,512}-Y^{14,512}}$} &
      \multicolumn{2}{c}{$\nm{ P^{m,512}- P^{14,512}}$} &
      \multicolumn{2}{c}{$\nm{U^{m,512}- U^{14,512}}$} \\
      \midrule
      \multirow{5}{*}{$\alpha=0.4$}
      &$ 8 $ &3.06e-4  & Order&  2.29e-4 &Order &  2.28e-4&  Order \\
      &$ 9 $ &1.54e-4 & 0.99 &   1.36e-4   & 0.75&   1.36e-4   & 0.75 \\
      &$ 10 $ &7.72e-5   & 1.00 & 7.87e-5&0.79 &   7.86e-5&  0.79 \\
      &$ 11 $ &3.85e-5  &1.00 & 4.40e-5 &0.84 &  4.39e-5& 0.84  \\
      &$ 12 $ &1.88e-5 &1.03&2.33e-5 &0.91&2.33e-5  & 0.91     \\
      \midrule
      \multirow{5}{*}{$\alpha=0.8$}
      &$ 8 $ &9.13e-4  & Order    & 2.02e-4 & Order  &  1.97e-4    & Order  \\
      &$ 9 $ &4.53e-4   & 1.01    & 1.01e-4&0.99   &  9.91e-5&  0.99 \\
      &$ 10 $ & 2.25e-4  & 1.01& 5.03e-5  &1.01   &   4.92e-5 & 1.01  \\
      &$ 11 $ & 1.11e-4   &1.02   &2.46e-5    &1.03&2.40e-5 & 1.03    \\
      &$12$ &  5.39e-5     & 1.04  & 1.17e-5 &1.07   &1.14e-5 & 1.07 \\
      \bottomrule
    \end{tabular}
  \end{table}

  \begin{table}[H]
    \centering
    \caption{Convergence history with $r=0.25$.}
    \label{tab:time4}
    \footnotesize
    \setlength{\tabcolsep}{6pt}
    \begin{tabular}{cccccccc}
      \toprule &$m$&
      \multicolumn{2}{c}{$\nm{ Y^{m,512}-Y^{14,512}}$} &
      \multicolumn{2}{c}{$\nm{ P^{m,512}- P^{14,512}}$} &
      \multicolumn{2}{c}{$\nm{U^{m,512}- U^{14,512}}$} \\
      \midrule
      \multirow{5}{*}{$\alpha=0.4$}
      &$ 8 $ &4.30e-4  & Order&  2.23e-4 &Order &  2.22e-4&  Order \\
      &$ 9 $ &2.33e-4 & 0.89 &   1.33e-4   & 0.75&   1.32e-4   & 0.75 \\
      &$ 10 $ &1.25e-4   & 0.90 & 7.67e-5&0.79 &   7.66e-5&  0.79 \\
      &$ 11 $ &6.57e-5  &0.92 & 4.29e-5 &0.84 &  4.28e-5& 0.84  \\
      &$ 12 $ &3.37e-5 &0.96&2.28e-5 &0.91&2.28e-5 & 0.91     \\
      \midrule
      \multirow{5}{*}{$\alpha=0.8$}
      &$ 8 $ &1.30e-3  & Order    & 1.95e-4 & Order  &  1.94e-4    & Order  \\
      &$ 9 $ &7.31e-4   & 0.84   & 9.82e-5&0.99   &  9.76e-5&  0.99 \\
      &$ 10 $ &4.05e-4  & 0.85& 4.88e-5  &1.01   &   4.85e-5 & 1.01  \\
      &$ 11 $ & 2.19e-4  &0.88   &2.38e-5    &1.03&2.38e-5 & 1.03    \\
      &$12$ &  1.14e-4     &0.95  & 1.13e-5 &1.08   &1.13e-5 & 1.08 \\
      \bottomrule
    \end{tabular}
  \end{table}

  \medskip\noindent\textbf{Experiment 3.} This experiment  investigates the temporal
  accuracy with uniform temporal grids. The reference solutions are $
  U^{14,512}$, $ Y^{14,512}$ and $P^{14,512}$, and we use $
  \widetilde U^{m,512}$, $\widetilde  Y^{m,512}$ and $\widetilde P^{m,512}$ to denote the corresponding numerical solutions of \cref{eq:optim_cond_full} with the
  mesh parameters $ M=2^m $, $ h=512 $ and  $\sigma_{1} = \sigma_{2}=1$. From \cref{tab:time1,tab:time2}, it is easy to see that the errors are generally larger than the cases with graded temporal grids.
  \begin{table}[H]
    \centering
    \caption{Convergence history with $r=0$.}
    \label{tab:time1}
    \footnotesize
    \setlength{\tabcolsep}{6pt}
    \begin{tabular}{cccccccc}
      \toprule &$m$&
      \multicolumn{2}{c}{$\nm{ \widetilde Y^{m,512}-Y^{14,512}}$} &
      \multicolumn{2}{c}{$\nm{ \widetilde P^{m,512}- P^{14,512}}$} &
      \multicolumn{2}{c}{$\nm{\widetilde U^{m,512}- U^{14,512}}$} \\
      \midrule
      \multirow{5}{*}{$\alpha=0.4$}
      &$ 8 $ &4.66e-3  & Order&  3.88e-4 &Order &  3.64e-4&  Order \\
      &$ 9 $ &3.35e-3 & 0.48 &   2.49e-4   & 0.64&   2.36e-4   & 0.63 \\
      &$ 10 $ &2.40e-3   & 0.48 & 1.56e-4&0.67&   1.49e-4&  0.66 \\
      &$ 11 $ &1.71e-3 &0.48 & 9.55e-5 &0.71&  9.19e-5& 0.70 \\
      &$ 12 $ &1.22e-3 &0.49 &5.72e-5 &0.74&5.53e-5  & 0.73     \\
      \midrule
      \multirow{5}{*}{$\alpha=0.8$}
      &$ 8 $ &1.11e-2  & Order    & 2.09e-4 & Order  &  1.97e-4    & Order  \\
      &$ 9 $ &7.86e-3   & 0.49   &1.05e-4&0.99   &  9.97e-5&  0.99 \\
      &$ 10 $ &5.57e-3  & 0.50& 5.21e-5  &1.01   &   4.95e-5 & 1.01  \\
      &$ 11 $ & 3.94e-3  &0.50   &2.55e-5    &1.03&2.42e-5 & 1.03    \\
      &$12$ &  2.78e-4     &0.50  & 1.22e-5 &1.06   &1.15e-5 & 1.07\\
      \bottomrule
    \end{tabular}
  \end{table}
  \begin{table}[H]
    \centering
    \caption{Convergence history with $r=0.25$.}
    \label{tab:time2}
    \footnotesize
    \setlength{\tabcolsep}{6pt}
    \begin{tabular}{cccccccc}
      \toprule &$m$&
      \multicolumn{2}{c}{$\nm{ \widetilde Y^{m,512}-Y^{14,512}}$} &
      \multicolumn{2}{c}{$\nm{ \widetilde P^{m,512}- P^{14,512}}$} &
      \multicolumn{2}{c}{$\nm{\widetilde U^{m,512}- U^{14,512}}$} \\
      \midrule
      \multirow{5}{*}{$\alpha=0.4$}
      &$ 8 $ &2.14e-3  & Order&  3.60e-4 &Order &  3.53e-4&  Order \\
      &$ 9 $ &1.45e-3 & 0.56 &   2.33e-4   & 0.63&   2.29e-4   & 0.62 \\
      &$ 10 $ &9.68e-4   & 0.58& 1.47e-4&0.67 &   1.45e-4&  0.66 \\
      &$ 11 $ &6.42e-4  &0.59 & 9.02e-5 &0.70 &  8.93e-5& 0.70  \\
      &$ 12 $ &4.24e-5 &0.60&5.42e-5 &0.73&5.37e-5 & 0.73     \\
      \midrule
      \multirow{5}{*}{$\alpha=0.8$}
      &$ 8 $ &2.45e-3  & Order    & 1.97e-4 & Order  &  1.94e-4    & Order  \\
      &$ 9 $ &1.51e-3   & 0.70  & 9.92e-5&0.99   &  9.78e-5&  0.99 \\
      &$ 10 $ &9.24e-4  & 0.71& 4.92e-5  &1.01   &   4.85e-5 & 1.01  \\
      &$ 11 $ & 5.62e-4  &0.72  &2.41e-5    &1.03&2.37e-5 & 1.03    \\
      &$12$ &  3.39e-4     &0.73  & 1.14e-5 &1.07   &1.13e-5 & 1.07 \\
      \bottomrule
    \end{tabular}
  \end{table}
  \appendix

	\section{Regularity of a fractional diffusion equation}\label{sec:regu-S}
	\subsection{Regularity in interpolation spaces}

	We first introduce the interpolation space theory (cf.~\cite[Chapter 2,
	pp.~54--55]{Lunardi2018}). Assume that $ (X,Y) $ is an interpolation couple of
	complex Banach spaces. For any $ 0 < \theta_1 < \theta_2 < 1 $, $ 0 < \theta < 1
	$ and $ 1 \leqslant q \leqslant \infty $,
	\begin{equation}
	\label{eq:complex_real}
	([X,Y]_{\theta_1}, [X,Y]_{\theta_2})_{\theta,q} =
	(X,Y)_{(1-\theta)\theta_1 + \theta\theta_2,q}
	\end{equation}
	with equivalent norms, where $[\cdot,\cdot]_{\theta}$ and $
	(\cdot,\cdot)_{\theta,q} $ denote the interpolation spaces defined by the
	complex method and the real method, respectively. For each $ w \in [X,Y]_\theta
	$ with $ 0 < \theta < 1 $,
	\begin{equation}
	\label{eq:Ktw}
	K(t,w) \leqslant 2t^\theta \nm{w}_{[X,Y]_\theta},
	\quad t > 0,
	\end{equation}
	where
	\[
	K(t,w) := \inf_{w=x+y,x\in X,y\in Y} \nm{x}_X + t \nm{y}_Y.
	\]
	If $ Y $ is continuously embedded into $ X $, then
	\begin{equation}
	\label{eq:65}
	[X,Y]_{\theta_2}  \text{ is continuously embedded into }
	[X,Y]_{\theta_1}
	\end{equation}
	for any $ 0 < \theta_1 < \theta_2 < 1 $.

	\begin{lemma}
		\label{lem:complex-real}
		If $ 0 < r < s < 1 $ and $ 1 \leqslant q < \infty $, then
		\begin{equation}
		\label{eq:complex-real}
		\nm{w}_{(X,Y)_{r,q}} \leqslant C \nm{w}_{[X,Y]_s}
		\end{equation}
		for all $ w \in [X,Y]_s $, where $ C $ is a positive constant independent of $
		w $.
	\end{lemma}
	\begin{proof}
		A straightforward calculation gives that
		\begin{align*}
		& \nm{w}_{(X,Y)_{r,q}}^q =
		\int_0^\infty \snm{t^{-r} K(t,w)}^q  \frac{\mathrm{d}t}{t} \\
		\leqslant{} & 2^q \int_0^1 \left( t^{s-r} \nm{w}_{[X,Y]_s} \right)^q
		\frac{\mathrm{d}t}t + 2^q
		\int_1^\infty \left( t^{-r/2} \nm{w}_{[X,Y]_{r/2}} \right)^q
		\frac{\mathrm{d}t}t \quad\text{(by \cref{eq:Ktw})} \\
		={} & \frac{2^q}{(s-r)q} \nm{w}_{[X,Y]_s}^q +
		\frac{2^{q+1}}{qr} \nm{w}_{[X,Y]_{r/2}}^q \\
		\leqslant{} & C \nm{w}_{[X,Y]_s}^q
		\quad\text{(by \cref{eq:65}),}
		\end{align*}
		where $ C $ is a positive constant independent of $ w $. This proves
		\cref{eq:complex-real} and hence this lemma.
	\end{proof}

	For $ m \in \mathbb N_{>0} $, $ 0 < \theta < 1 $ and $ 1 \leqslant q < \infty $,
	define
	\[
	{}_0H^{m\theta,q}(0,T;X) := \big[
	L^q(0,T;X), \, {}_0W^{m,q}(0,T;X)
	\big]_{\theta},
	\]
	where $ X $ is a Hilbert space. We  use $ {}_0H^{0,q}(0,T;X) $ to denote the
	space $ L^q(0,T;X) $. For $ 0 < \beta < \infty $ and $ 1 \leqslant q < \infty $,
	we have the following properties: if $ \beta \in \mathbb N_{>0} $ then
	\begin{equation}
	\label{eq:H-W-1}
	{}_0H^{\beta,q}(0,T;X) = {}_0W^{\beta, q}(0,T;X)
	\quad\text{ with equivalent norms;}
	\end{equation}
	if $ q = 2 $ then (cf.~\cite[Corollary 4.37]{Lunardi2018})
	\begin{equation}
	\label{eq:H-W-2}
	{}_0H^{\beta,q}(0,T;X) = {}_0W^{\beta,q}(0,T;X)
	\quad\text{ with equivalent norms.}
	\end{equation}

	By \cite[Theorem 4.5.15]{Pruss2016}, for any $ g \in
	{}_0H^{\beta,q}(0,T;L^2(\Omega)) $ with $ 0 \leqslant \beta < \infty $ and $ 1 <
	q < \infty $, there exists a unique $ Sg \in
	{}_0H^{\alpha+\beta,q}(0,T;L^2(\Omega)) \bigcap {}_0H^{\beta,q}(0,T;\dot
	H^2(\Omega)) $ such that
	\[
	(\D_{0+}^\alpha - \Delta) Sg = g;
	\]
	moreover,
	\begin{equation}
	\label{eq:Sg-complex}
	\nm{Sg}_{{}_0H^{\alpha+\beta,q}(0,T;L^2(\Omega))} +
	\nm{Sg}_{{}_0H^{\beta,q}(0,T;\dot H^2(\Omega))}
	\leqslant C_{\alpha,\beta,q}
	\nm{g}_{{}_0H^{\beta,q}(0,T;L^2(\Omega))}.
	\end{equation}

	\begin{lemma}
		\label{lem:regu-lp}
		Assume that $ 1 < q < \infty $. If $ g \in L^q(0,T;L^2(\Omega)) $, then
		\begin{equation}
		\label{eq:lbj-3}
		\nm{Sg}_{{}_0W^{\beta,q}(0,T;L^2(\Omega))}
		\leqslant C_{\alpha,\beta,q} \nm{g}_{L^q(0,T;L^2(\Omega))}
		\end{equation}
		for all $ 0 < \beta < \alpha $. If $ g \in {}_0W^{\beta,q}(0,T;L^2(\Omega)) $
		with $ 0 < \beta < \infty $, then
		\begin{equation}
		\label{eq:lbj-2}
		\nm{Sg}_{{}_0W^{\beta,q}(0,T;\dot H^2(\Omega))}
		\leqslant C_{\alpha,\beta,q} \nm{g}_{{}_0W^{\beta,q}(0,T;L^2(\Omega))}.
		\end{equation}
		If $ g \in {}_0W^{\beta,q}(0,T;L^2(\Omega)) $ with $ 0 < \beta < \infty $ and
		$ \alpha + \beta \not\in \mathbb N $, then
		\begin{equation}
		\label{eq:lbj-1}
		\nm{Sg}_{{}_0W^{\alpha+\beta,q}(0,T;L^2(\Omega))}
		\leqslant C_{\alpha,\beta,q} \nm{g}_{{}_0W^{\beta,q}(0,T;L^2(\Omega))}.
		\end{equation}
	\end{lemma}
	\begin{proof}
		Estimate \cref{eq:lbj-3} follows from \cref{lem:complex-real,eq:Sg-complex},
		and estimate \cref{eq:lbj-2} follows from \cref{eq:H-W-1},
		\cref{eq:Sg-complex} and \cite[Theorem 1.6]{Lunardi2018}. Let us proceed to
		prove \cref{eq:lbj-1} for the case $ \beta \in (0,1] \setminus \{1-\alpha\} $.
		By \cref{eq:Sg-complex} it holds that
		\[
		\nm{Sw}_{{}_0H^{\alpha,q}(0,T;L^2(\Omega))}
		\leqslant C_{\alpha,q} \nm{w}_{L^q(0,T;L^2(\Omega))}
		\]
		for all $ w \in L^q(0,T;L^2(\Omega)) $ and that
		\[
		\nm{Sw}_{{}_0H^{\alpha+1,q}(0,T;L^2(\Omega))}
		\leqslant C_{\alpha,q} \nm{w}_{{}_0W^{1,q}(0,T;L^2(\Omega))}
		\]
		for all $ w \in {}_0W^{1,q}(0,T;L^2(\Omega)) $. Hence, applying \cref{eq:complex_real} and the real
		interpolation of type $ (\beta,q) $ yields that
		\begin{align*}
		\nm{Sw}_{{}_0W^{\alpha+\beta,q}(0,T;L^2(\Omega))}
		\leqslant C_{\alpha,\beta,q}
		\nm{w}_{{}_0W^{\beta,q}(0,T;L^2(\Omega))}
		\end{align*}
		for all $ w \in {}_0W^{\beta,q}(0,T;L^2(\Omega)) $. For $ 1 < \beta < \infty $
		and $ \alpha+\beta \not\in \mathbb N $, \cref{eq:lbj-1} can be proved
		analogously. This completes the proof.
	\end{proof}
	\begin{remark}
		In the above proof of \cref{eq:lbj-1}, we have used the following well-known
		result (cf.~\cite[Proposition 3.8]{Lunardi2018}): if $ j, k, m, n \in \mathbb
		N $ and $ 0 < r,s < 1 $ satisfy that $ j < k $, $ m < n $ and $ (1-r)j+r k =
		(1-s)m+sn \not\in \mathbb N $, then
		\begin{align*}
		& \big(
		{}_0W^{j,q}(0,T;L^2(\Omega)),\,
		{}_0W^{k,q}(0,T;L^2(\Omega))
		\big)_{r,q} \\
		={} & \big(
		{}_0W^{m,q}(0,T;L^2(\Omega)),\,
		{}_0W^{n,q}(0,T;L^2(\Omega))
		\big)_{s,q},
		\end{align*}
		with equivalent norms, where $ 1 \leqslant q < \infty $.
	\end{remark}
	\subsection{ Regularity from the Mittag-Leffler function}
	For any $ \beta, \gamma > 0 $, define the Mittag-Leffler function $
	E_{\beta,\gamma} $ by that
	\begin{equation}
	\label{eq:ml}
	E_{\beta,\gamma}(z) :=
	\sum_{k=0}^\infty \frac{z^k}{\Gamma(k\beta + \gamma)},
	\quad z \in \mathbb C.
	\end{equation}
	This function has a useful growth estimate \cite{Podlubny1998}:
	\begin{equation}
	\snm{E_{\beta,\gamma}(-t)} \leqslant
	\frac{C_{\beta,\gamma}}{1+t}, \quad t > 0.
	\end{equation}
	Moreover, the Mittag-Leffler function admits the asymptotic expansion (cf. \cite[pp. 207]{function}):
	\[
	E_{\beta,1}(-t) = \sum _{k=1}^N \frac{(-1)^{k+1}t^{-k}}{\Gamma(1-k\beta)} + O(t^{-N-1}), \quad \text{as} \ t \rightarrow \infty.
	\]
	A straightforward calculation gives that \cite{Sakamoto2011}, for any $ v \in
	L^2(\Omega) $,
	\begin{align}
	& (Sv)(t) = t^\alpha \sum_{k=0}^\infty
	E_{\alpha,1+\alpha}(-\lambda_k t^\alpha)
	\dual{v,\phi_k}_\Omega \, \phi_k,
	\label{eq:Sv} \\
	& (S\D_{0+}^\alpha v)(t) =
	\sum_{k=0}^\infty E_{\alpha,1}(-\lambda_k t^\alpha)
	\dual{v,\phi_k}_\Omega \, \phi_k,
	\label{eq:SDalpha_v}
	\end{align}
	for each $ 0 \leqslant t \leqslant T $. Here
	\begin{equation}\label{basis}
	\{\phi_k: k \in \mathbb N \} \subset H_0^1(\Omega) \cap H^2(\Omega)
	\end{equation}
	is an orthonormal basis of $ L^2(\Omega) $ such that   for all $ k \in \mathbb N $,
	\begin{equation}\label{eigen}
	-\Delta \phi_k = \lambda_k \phi_k \quad \text{ in  }   \Omega,
	\end{equation}
	where $ \{ \lambda_k: \ k\in \mathbb N \} \subset \mathbb R_{>0} $ is a
	non-decreasing sequence.

	By the  above properties of Mittag-Leffler function, \cite[Lemma 3.4]{Luo2019} and some techniques in \cite{mclean2010regularity},  a few straightforward
	calculations yield the following two lemmas. 

	\begin{lemma}
		\label{lem:grow}
Assume that $ 0 < t < T $ and $ v \in L^2(\Omega) $ then
			\begin{align}
			t^{-\alpha} \nm{(Sv)(t)}_{L^2(\Omega)} +
			t^{1-\alpha} \nm{(Sv)'(t)}_{L^2(\Omega)}
			& \leqslant C_\alpha \nm{v}_{L^2(\Omega)},
			\label{eq:(Sv)'-growth} \\
			(T-t)^{-\alpha} \nm{(S^*v)(t)}_{L^2(\Omega)} +
			(T-t)^{1-\alpha} \nm{(S^*v)'(t)}_{L^2(\Omega)}
			&  \leqslant C_\alpha \nm{v}_{L^2(\Omega)}.
			\label{eq:(S^*v)'-growth}
			\end{align}
Moreover, if $ v \in \dot H^{2r }(\Omega) $ with $0<r<1$ then
		\begin{small}
			\begin{align}
			\nm{(S\D_{0+}^\alpha v)(t)}_{L^2(\Omega)}
			\leqslant C_\alpha t^{\alpha r} \nm{v}_{\dot H^{2r}(\Omega)}.
			\label{eq:SDalpha_v_l2} \\
			\label{eq:SDalpha_v_h1}
			\nm{(S\D_{0+}^\alpha v)'(t)}_{L^2(\Omega)} +
			t^{\alpha/2} \nm{(S\D_{0+}^\alpha v)'}_{\dot H^1(\Omega)}
			\leqslant C_\alpha t^{\alpha r-1} \nm{v}_{\dot H^{2r}(\Omega)}.
			\end{align}
		\end{small}
	\end{lemma}

  \begin{lemma} \label{lem:regu2}
    Assume that $ v \in \dot H^{2r}(\Omega) $ with $0<r<1$. If $ 0<\alpha<1/3$, then
    \begin{align*}
      \nm{S \D_{0+}^{\alpha} v}_{
        {}_0H^{\alpha/2}(0,T;\dot H^{2+2r}(\Omega))
      } \leqslant C_{\alpha,r,\Omega,T}
      \nm{v}_{\dot H^{2r}(\Omega)}.
    \end{align*}
    If $\alpha=1/3$, then for any $ 0 < \epsilon < 1/2 $,
    \begin{align*}
      \nm{S \D_{0+}^{\alpha} v}_{
        {}_0H^{\alpha/2}(0,T;\dot H^{2+2r-\epsilon}(\Omega))
      } \leqslant C_{\alpha,r,\Omega,T}
      \epsilon^{-1/2} \nm{v}_{\dot H^{2r}(\Omega)}.
    \end{align*}
    If $ 1/3<\alpha<1$, then
    \begin{align*}
      \nm{S \D_{0+}^{\alpha} v}_{
        {}_0H^{\alpha/2}(0,T;\dot H^{1/\alpha+2r-1}(\Omega))
      } \leqslant C_{\alpha,r,\Omega,T}
      \nm{v}_{\dot H^{2r}(\Omega)}.
    \end{align*}
    If		$ 0<\alpha<1/2$, then
    \begin{align*}
      \nm{S \D_{0+}^{\alpha} v}_{L^2(0,T;\dot H^{2+2r}(\Omega))}
      \leqslant C_{\alpha,r,\Omega,T} \nm{v}_{\dot H^{2r}(\Omega)}.
    \end{align*}
    If 	$\alpha=1/2$, then for any $ 0 < \epsilon < 1/2$,
    \begin{align*}
      \nm{S \D_{0+}^{\alpha} v}_{
        L^2(0,T;\dot H^{2+2r-\epsilon}(\Omega))
      } \leqslant C_{\alpha,r,\Omega,T}
      \epsilon^{-1/2} \nm{v}_{\dot H^{2r} (\Omega)}.
    \end{align*}
    If $ 1/2 < \alpha < 1 $, then
    \begin{align*}
      \nm{S \D_{0+}^{\alpha} v}_{
        L^2(0,T;\dot H^{2r+1/\alpha}(\Omega))
      } \leqslant C_{\alpha,r,\Omega,T}
      \nm{v}_{\dot H^{2r}(\Omega)}.
    \end{align*}
  \end{lemma}

\end{document}